\documentclass[11pt]{article}

\usepackage{amsmath,amsfonts,amsthm,amsrefs}

\usepackage{graphicx,color,pict2e}

\definecolor{refkey}{rgb}{1, 0.5, 0}  
\definecolor{labelkey}{rgb}{0.1,1,0.2}

\newtheorem{theorem}{Theorem}[section]
\newtheorem{lemma}{Lemma}[section]

\newtheorem{corollary}{Corollary}[section]
\newtheorem{definition}{Definition}[section]

\def\implies{\Longrightarrow}
\def\bel{\begin{equation}\label}
\def\eeq{\end{equation}}
\def\bega{\begin{array}}
\def\enda{\end{array}}

\def\O{\mathcal{O}}

\def\dx{{\Delta x}}

\def\cW{\mathcal{W}}

\title{Traveling Waves for Conservation Laws with Nonlocal Flux for Traffic Flow on 
Rough Roads}
\author{Wen Shen\\ Mathematics Department, Penn State University, USA \\ Email: wxs27@psu.edu}

\begin{document}

\maketitle

\begin{abstract}
We consider two scalar conservation laws with non-local flux functions, describing 
traffic flow on roads with rough conditions. 
In the first model, the velocity of the car depends on an averaged downstream density,
while in the second model one considers an averaged downstream velocity. 
The road condition is piecewise constant with a jump at $x=0$. 
We study stationary traveling wave profiles cross $x=0$,
for all possible cases. 
We show that, depending on the case, there could exit infinitely many profiles,
a unique profile, or no profiles at all. 
Furthermore, some of the profiles are time asymptotic solutions for the Cauchy problem
of the conservation laws under mild assumption on the initial data,
while other profiles are unstable.  
\end{abstract}

{\small 

}

%%%%%%%%%%%%%%%%
\section{Introduction}\label{sec:1}
\setcounter{equation}{0}

We consider two scalar conservation laws with nonlocal flux, 
describing traffic flow with varying road condition.  
To be specific, we consider the integro-differential equations
\begin{equation}\label{eq:CL1}
\text{(M1)} \hspace{2cm}
\rho_t + \left[ \rho(t,x)   \kappa(x) \cdot v\left(\int_x^{x+h} \rho(t,y) w(y-x) \; dy\right)\right]_x =0, 
\end{equation}
and
\begin{equation}\label{eq:CL2}
\text{(M2)} \hspace{2cm}
\rho_t + \left[ \rho(t,x)   \cdot \int_x^{x+h} \kappa(y) v( \rho(t,y)) w(y-x) \; dy\right]_x =0.
\end{equation}
In both equations,  $\rho$ is the main unknown, denoting the density of cars on a road.
When $\rho=0$, the road is empty; when $\rho=1$, the road is packed full. 
The function $v(\cdot)\in \mathbb{C}^2$ satisfies
\begin{equation}\label{eq:vc}
v(0)=1, \quad v(1)=0, \quad 
\mbox{and}\quad v'(\rho) <0, \quad v''(\rho)\le 0 \quad \forall \rho\in[0,1].
\end{equation}
For example, the classical Lighthill-Whitham model~\cite{MR0072606} uses 
$v(\rho)=1-\rho$. 

The function $\kappa(x)$ denotes the speed limit of the road at $x$, which
represents the road condition. 
We consider rough road condition where $\kappa(x)$ can be discontinuous.

The models (M1) and (M2) can be formally derived as the continuum limit of particle models
where the speed of each car follows certain rules. 
For  model (M1), the speed of the car at $x$ is 
\[ \kappa(x) \cdot v\left(\int_x^{x+h} \rho(t,y) w(y-x) \,dy\right),\]
where a weighted average of car density 
on an interval of length $h$ in front of the driver (i.e., downstream) 
is computed. 
For  model (M2), the speed of the car at $x$ is 
\[ \int_x^{x+h} \kappa(y) v( \rho(t,y)) w(y-x) \; dy,\]
where the  weighted average is taken over $\kappa(y)v(\rho(t,y))$
over an interval of length $h$ in front of the driver.

The weight $w(x)$ is a non-negative function defined on $x\ge 0$.  
For a given $h$, we assume that $w(x)$ is continuous and bounded 
on $x>0$ with its support on $[0,h)$, and satisfies
\begin{equation}\label{eq:wc1}
\int_0^h w(x)\, dx=1, \qquad 
w(x) =0 \quad \forall  x\ge h, 
\qquad  w'(x) <0 \quad \forall x \in(0,h).
\end{equation}
Here, the assumption $w'(x) <0$  indicates that the condition right in front of the driver
is more important than those further ahead.

Formally, as $h\to 0+$ and $w(\cdot)$ converges to a dirac delta, 
both (M1) and (M2) converge to a scalar conservation law with local flux
\begin{equation}\label{eq:claw}
\rho_t + f(\kappa(x),\rho)_x=0, \qquad   \mbox{where}
\quad f(\kappa,\rho) = \kappa(x) \cdot \rho \;v(\rho).
\end{equation}
Unfortunately rigorous proofs of such convergences have not been established yet 
in the literature, not even for the case where $\kappa(x)$ is a constant function.

%Some discussions ... 
Non-local conservation laws has gained growing interests in recent years. 
In the simpler case where $\kappa(x)\equiv1$,  the existence and well-posedness of
solutions of the Cauchy problem for (M1) were established in~\cite{BG2016},
using numerical approximations generated by the Lax-Friedrich scheme. 
The same results for (M2) were proved in~\cite{Friedrich2018},
using numerical approximations generated by a Godunov-type scheme. 
Other models of conservation laws with nonlocal flux include 
models for slow erosion of granular flow~\cite{AmadoriShen2012, ShenZhang},
for synchronization~\cite{AmadoriHaPark2017}, 
for sedimentation~\cite{BetancourtBuergerKarlsenTory2011}, 
for nonlocal crowd dynamics~\cite{ColomboLecureuxMercier2011},
and for material with fading memory~\cite{Chen2007}.
Models with symmetric kernel functions have been studied in~\cite{Zumbrun1999},
and for systems in several space dimensions~\cite{AggarwalColomboGoatin2015}.
Some numerical integrations are studied in~\cite{AmorimColomboTeixeira2015}, 
and an overview of 
several nonlinear nonlocal models can be found in~\cite{DuKammLehoucqParks2012}.

In this paper we are interested in the traveling wave profiles for (M1) and (M2). 
We remark that, when the road condition is uniform, say $\kappa(x) \equiv 1$, 
the traveling wave profiles for (M1) and (M2) are studied in a recent 
work~\cite{RidderShen2018},
where we established various results on existence, 
uniqueness and stability of the traveling wave profiles. 
In this paper, however, we consider rough road condition, 
where $\kappa(x)$ is discontinuous.
To fix the idea,  we consider
\begin{equation}\label{kappa}
\kappa(x) = \begin{cases} \kappa^-, & x<0,\\
\kappa^+, & x>0 .
\end{cases}
\end{equation}

The main objective of this paper is to study the stationary traveling wave  profiles 
of  (M1) and (M2),  crossing the jump in $\kappa(x)$ at $x=0$. 
For all possible cases, we show analytical and numerical results 
on existence (and non-existence), uniqueness (and non-uniqueness) 
and stability (and instability) of these traveling wave profiles.

Traveling wave profiles for a local follow-the-leader model for traffic flow was studied 
in~\cite{ShenKarim2017} for homogeneous road conditions, 
and in ~\cite{ShenDDDE2017} with rough road condition. 
We also mention that, for the non-local models for slow erosion of granular flow, 
traveling waves and their local stability were studied in~\cite{GS2014}. 

The rest of the paper is organized as follows. 
In Section~\ref{sec:M1} we consider model (M1) and analyze two cases with 
$\kappa^- > \kappa^+$ and $\kappa^- < \kappa^+$, where 
each case has 4 sub-cases.  
In Section~\ref{sec:M2} we analyze model (M2).
Final remarks are given in Section~\ref{sec:fr}.

%%%%%%%%%%%%%%%%%%%%%%%%%%%
\section{Stationary traveling wave profiles for (M1)}\label{sec:M1}
\setcounter{equation}{0}

We seek a stationary profile $Q(x)$ for (M1) around $x=0$.
To simplify the notation, 
we introduce an averaging operator
\[%\begin{equation}\label{eq:A}
A(Q;x) \;\dot=\; \int_x^{x+h} Q(y) w(y-x) \; dy = \int_0^h Q(x+s) w(s) \; ds.
\]%\end{equation}
Note that, since $w(x)$ vanishes outside the interval $(0,h)$, we could 
put the upper integration bound to be $\infty$.
Furthermore,  as long as $Q(x)$ is bounded, 
$x\mapsto A(Q;x)$ is  Lipschitz continuous
even if $Q(x)$ is discontinuous. 

Since $Q(x)$ is a stationary solution for (M1), it must satisfy
\begin{equation}\label{eq:Q1}
Q(x)   \kappa(x) \cdot v(A(Q;x)) \equiv\bar f = \mbox{constant},
\end{equation}
where 
\[
\bar f =\lim_{x\to\pm\infty} Q(x)   \kappa(x) \cdot v(A(Q;x)) .
\]

In the case when 
\begin{equation}\label{eq:lim}
\lim_{x\to-\infty}Q(x) =\rho^-,\qquad 
\lim_{x\to+\infty}Q(x) =\rho^+,\qquad
\lim_{x\to\pm\infty}Q'(x)=0,
\end{equation}
the following constraint on $\rho^-,\rho^+$ must be imposed
\begin{equation}\label{eq:bf}
\bar f = f(\kappa^-,\rho^-) = f(\kappa^+,\rho^+). 
\end{equation}

Differentiating~\eqref{eq:Q1} in $x$, we obtain an delay integro-differential equation
\begin{eqnarray}
&& \hspace{-2cm} Q'(x-)   \kappa(x-)  v(A(Q;x)) \nonumber\\
&=&-Q(x) \left[ \delta_0(x) (\kappa^+-\kappa^-) v(A(Q;x)) + k(x) v'(A(Q;x)) A(Q; x)_x\right].
\label{eq:Q2}
\end{eqnarray}
Here, $\delta_0(x)$ denote the dirac delta. 
Since $\kappa(x)$ and $Q(x)$ may be discontinuous,
in the left-hand side of the equation  the quantities are evaluated at $(x-)$.
For general theory on delay differential equations we refer to~\cite{MR0141863, MR0477368}.

The profiles $Q(x)$ are discontinuous at $x=0$. 
%We denote the traces along $x=0$ as
%\[%\begin{equation}\label{eq:Q+-}
%Q^- =Q(0-), \qquad Q^+ = Q(0+).
%\]%\end{equation}
Since $A(Q;x)$ is continuous at $x=0$, so must $\kappa(x)Q(x)$.  
We impose the {\em connecting condition} on the traces
\begin{equation}\label{eq:connect}
Q(0-) \kappa^-= Q(0+)\kappa^+.
\end{equation}
Thus, the jump $Q(x)$ at $x=0$ is in the opposite direction of the jump in $\kappa(x)$.
In summary, the equation~\eqref{eq:Q1} with $\kappa(x)$ in~\eqref{kappa}
is equivalent to 
\begin{equation}\label{eq:Q3}
Q(x) v(A(Q;x)) = \begin{cases} 
\bar f / \kappa^-, & (x<0) \\
\bar f / \kappa^+, & (x>0) 
\end{cases} 
\qquad \mbox{and}\quad 
Q(0-) \kappa^-= Q(0+)\kappa^+. 
\end{equation}

%%%%%%%%%%%%%%%%%%%%%%%%%%%
\subsection{Review of previous results} \label{sec:review}

The simpler case where $\kappa(x)\equiv \bar\kappa=$ constant
is studied in a recent work~\cite{RidderShen2018}.
We review some related results which will be useful in the analysis of this paper. 

Under the assumption~\eqref{eq:vc}, 
there exists a unique stagnation point $\hat\rho$ where
\begin{equation}\label{eq:stag}
f_\rho(\kappa,\hat \rho)=0 ~~\forall \kappa, \qquad 
f_\rho(\kappa,\rho)>0 ~\mbox{for} ~\rho<\hat\rho, 
\quad f_\rho(\kappa,\rho)<0 ~\mbox{for} ~\rho>\hat\rho.
\end{equation}

Let $W(x)$ be a stationary profile for (M1) with $\kappa(x)\equiv \bar\kappa$.
It satisfies the following  integro-equation
\begin{equation}\label{eq:OW}
W(x) \cdot v\left(\int_x^{x+h} \rho(t,y) w(y-x) \; dy\right) \equiv 
\frac{\bar f }{\bar\kappa} = \frac{1}{\bar\kappa} \lim_{x\to\pm\infty} f(\bar\kappa,W(x)).
\end{equation}

Next Lemma was proved in~\cite{RidderShen2018} (Lemma~3.1). 

%%%%%%%%%%%%
\begin{lemma}[Asymptotic limits]\label{lm:AL}
Assume that $W(x)$ is a solution of~\eqref{eq:OW} which satisfies
\[ %\begin{equation}\label{eq:Lim}
\lim_{x\to-\infty} W(x) = \rho^-,\qquad 
\lim_{x\to +\infty} W(x) = \rho^+,\qquad \lim_{x\to\pm\infty} W'(x)=0.
\] %\end{equation}
Let $\hat\rho$ be the stagnation point satisfying~\eqref{eq:stag}. 
The following holds.
\begin{itemize}
\item[i)] As $x\to+\infty$, $W(x)$ approaches $\rho^+$ with an exponential rate 
if and only if  $\rho^+ > \hat\rho$;
\item[ii)] As $x\to-\infty$, $W(x)$  approaches $\rho^-$ with an exponential rate 
if and only if  $\rho^- < \hat\rho$. 
\end{itemize}
\end{lemma}

The following  existence and uniqueness result 
of the profile is proved in Theorem~3.2 of~\cite{RidderShen2018}.

%%%%%%%%%%%%
\begin{theorem}\label{th:O1}
Let $W(x)$ satisfy the equation~\eqref{eq:OW} and the boundary conditions
\begin{equation}\label{eq:BCO}
\lim_{x\to\pm\infty} W(x) = \rho^\pm, \quad
f(\bar\kappa,\rho^-) = f(\bar\kappa,\rho^+), 
\quad  \rho^- < \hat\rho<\rho^+.
\end{equation}
There exist  solutions of $W(x)$  which are monotone increasing.
Furthermore, the solution is unique up to horizontal shifts. 
\end{theorem}

\medskip

In the remaining of this section  we let $W(x)$ denote 
the smooth and monotone profile with  $\kappa(x)\equiv \kappa^+$ 
and the boundary conditions
\[
\lim_{x\to\infty}W(x) =\rho^+, \quad 
\lim_{x\to-\infty} W(x) = \rho^-_*, \]
where
\[ 0<\rho^-_* <\hat\rho<\rho^+<1, \qquad
f(\kappa^+, \rho^-_* ) = f(\kappa^+, \rho^+). 
\]
%By Theorem~\ref{th:O1}, such profiles are smooth and monotone increasing.

By Lemma~\ref{lm:AL}, we see that 
the asymptote $\lim_{x\to\infty} W(x) = \rho^+$ is stable if $\rho^+ > \hat\rho$, 
while the asymptote $\lim_{x\to-\infty} W(x) = \rho^-$ is stable if $\rho^- < \hat\rho$. 
This implies that $\rho^+ \le \hat\rho$ is an unstable asymptote for $x\to\infty$,
and $\rho^-\ge \hat\rho $  is an unstable asymptote for $x\to-\infty$.

\bigskip

When $\kappa(x)$ is discontinuous, the profiles $Q(x)$ are very differently.
We discuss two cases in the next two sections, 
where Case A is for $\kappa^->\kappa^+$, 
and Case B for $\kappa^-<\kappa^+$.

For notational convenience, we introduce the functions:
\begin{equation}\label{eq:ff}
f^-(\rho) \;\dot=\; f(\kappa^-,\rho)=\kappa^- \rho v(\rho),\qquad
f^+(\rho) \;\dot=\; f(\kappa^+,\rho)=\kappa^+ \rho v(\rho).
\end{equation}
Under the assumptions~\eqref{eq:vc}, both $f^-(\rho)$ and $f^+(\rho)$
are strictly concave functions for $\rho\in[0,1]$
with $f'' <0$ and $f^\pm(0)=f^\pm(1)=0$.

%%%%%%%%%%%%%%%%%%%%%%
\subsection{Case A:  $\kappa^- > \kappa^+$}
\label{sec:A}

The values $\rho^-,\rho^+$ must satisfy~\eqref{eq:bf}. 
Given $\bar f$ that lies in the ranges of both $f^-$ and $f^+$, 
there exist unique values $\rho_{1},\rho_2,\rho_3,\rho_4$ 
such that (see Figure~\ref{fig:rhos})
\begin{equation}\label{eq:cA}
\rho_1 < \rho_2 \le \hat \rho \le \rho_3 < \rho_4, \qquad 
f^-(\rho_1)=f^-(\rho_4) =\bar f = f^+(\rho_2)=f^+(\rho_3).
\end{equation}
%See Figure~\ref{fig:rhos} for an illustration. 

%%%%%%%%%%
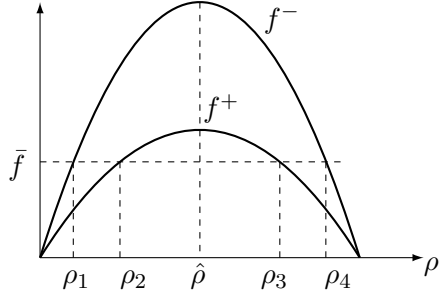
\begin{figure}[htbp]
\begin{center}
\setlength{\unitlength}{0.85mm}
\begin{picture}(60,50)(-3,-5)  % -- Left ---
\put(0,0){\vector(1,0){60}}\put(60,-2){$\rho$}
\put(0,0){\vector(0,1){40}}
\multiput(0,15)(2,0){24}{\line(1,0){1}}\put(-5,13){$\bar f$}
\multiput(5.2,15)(0,-2){8}{\line(0,-1){1}}\put(2,-4){ $\rho_1$}
\multiput(12.5,15)(0,-2){8}{\line(0,-1){1}}\put(11,-4){ $\rho_2$}
\multiput(25,40)(0,-2){20}{\line(0,-1){1}}\put(22,-4){ $\hat\rho$}
\multiput(37.5,15)(0,-2){8}{\line(0,-1){1}}\put(33,-4){ $\rho_3$}
\multiput(44.7,15)(0,-2){8}{\line(0,-1){1}}\put(43,-4){ $\rho_4$}
\put(25.5,22){$f^+$}\put(35,36){$f^-$}
\thicklines
\qbezier(0,0)(25,40)(50,0)
\qbezier(0,0)(25,80)(50,0)
\end{picture}
\caption{Flux functions $f^-(\rho), f^+(\rho)$, 
and location of $\rho_1,\rho_2,\rho_3,\rho_4$ and $\hat\rho$.}
\label{fig:rhos}
\end{center}
\end{figure}

We have 4 sub cases:
\begin{itemize}
\item[A1.]  We have $\rho^-=\rho_1$, $\rho^+=\rho_3$, with
$\rho^- < \hat \rho < \rho^+$; 
\item[A2.]  We have 
$\rho^-=\rho_1$, $\rho^+=\rho_2$, with $\rho^- <\rho^+ \le \hat\rho$;
\item[A3.]  We have $\rho^-=\rho_4$, $\rho^+=\rho_3$, with
$\hat\rho \le \rho^+ < \rho^-$; 
\item[A4.] We have $\rho^-=\rho_4$, $\rho^+=\rho_2$, with
$\rho^+ < \hat\rho <\rho^-$. 
\end{itemize}

We first observe that the cases with $\bar f =0$ are trivial. 
In this case we have $\rho_1=\rho_2=0$ and $\rho_3=\rho_4=1$. Then
\begin{itemize}
\item For Case A1, we $\rho^-=0,\rho^+=1$ and $Q(x)$
 is the unit step function;% with a jump at $x=0$;
\item For Case A2, we  have $\rho^-=\rho^+=0$, and we have $Q(x) \equiv 0$; 
\item for Case A3, we have $\rho^-=\rho^+=1$,  and we have $Q(x)\equiv 1$;
\item For Case A4, we $\rho^-=1,\rho^+=0$, there are no profiles. 
See discussion in Section~\ref{sec:A4}.
\end{itemize}
For the rest of the paper 
we only consider the nontrivial case with $\bar f >0$, where $0< \rho^\pm<1$.

%%%%%%%%%%%
\subsubsection{Case A1:  $0<\rho^-<\hat\rho<\rho^+<1$}
\label{sec:A1}

%We observe that,  
%if $\rho^-=0,\rho^+=1$ (i.e. $\bar f=0$)
%we have a trivial solution, which is the unit step function with a jump at $x=0$. 
%Below we consider the non-trivial case with
% \begin{equation}\label{A1cc}
% 0< \rho^-<\hat\rho<\rho^+< 1.  
% \end{equation}
 
\begin{theorem}[Existence of profiles]\label{th1}%%%%%%%%%%%%%
Assume $\kappa^->\kappa^+$, and 
let $\rho^-,\rho^+$ satisfy %~\eqref{A1cc}
$0<\rho^-<\hat\rho<\rho^+<1$
and~\eqref{eq:bf}. 
Then, there exist infinitely many  stationary monotone profiles $Q(x)$ 
which satisfies the equation~\eqref{eq:Q1} and the boundary conditions
\begin{equation}\label{eq:BC}
\lim_{x\to-\infty} Q(x) = \rho^-,\qquad \lim_{x\to\infty} Q(x) = \rho^+.
\end{equation}
The profiles  are discontinuous at $x=0$, but are continuous on $x>0$ and $x<0$.
\end{theorem}

\begin{proof}%%%%%%%%%%
The proof takes several steps.

\textbf{Step 1.} 
Since $\rho^+>\hat\rho$ is a stable asymptote as $x\to\infty$, then on 
 $x>0$,  the profile $Q(x)$ can be either the constant function $Q(x)\equiv \rho^+$, 
or some horizontal shifts of $W(x)$. 
This solution is  continuous and monotone increasing on $x>0$. 

\textbf{Step 2.}
At $x=0$ the profile $Q(x)$ has an upward jump. 
The traces $Q(0-)$ and $Q(0+)$ satisfy
\[
0 <Q(0-) = \frac{\kappa^+}{\kappa^-} Q(0+) < Q(0+)\le \rho^+<1.
\]

Furthermore, 
recalling the definition of $\rho_2$ in~\eqref{eq:cA}, we have
\[ \rho_2 <\hat \rho\quad \mbox{and}\quad 
\kappa^+ \rho_2 v(\rho_2) = \kappa^- \rho^- v(\rho^-).\]
Then it holds 
\[ %\begin{equation}\label{eq:cA2}
Q(0-) > \frac{\kappa^+}{\kappa^-} \rho_2 
= \frac{ \kappa^- \rho^- v(\rho^-)}{\kappa^- v(\rho_2)} 
= \frac{  v(\rho^-)}{ v(\rho_2)}  \rho^- > \rho^-.
\]%\end{equation}

\textbf{Step 3.}
With this given $Q(x)$ on $x\ge 0$,  we solve an ``initial value problem'' for $Q(x)$
backward in $x$ for  $x<0$. 
In order to establish the existence of solutions for the initial value problem, 
we generate approximate  solutions by discretization. 
Fix the mesh size $\dx$, 
we have the discretization
%we generate a uniform mesh on $x\le 0$ by setting
\begin{equation}\label{eq:mesh}
x_i = i \dx, \qquad Q_i \approx Q(x_i), \quad i=\mathbb{Z}^-, 
\qquad \mbox{and}\quad  Q_0 = Q(0-). 
\end{equation}
On $x<0$, the approximate solution $Q^\dx(x)$ is reconstructed 
as the linear interpolation through
the discrete values $Q_i$, i.e.,
\[
Q^\dx(x) =  Q_{i-1} \frac{x_i-x}{\dx}  +  Q_{i}\frac{x-x_{i-1}}{\dx} ,
\qquad \mbox{for}\quad x\in[x_{i-1},x_i].
\]

The discrete values $Q_i$ can be generated iteratively. 
Given a profile 
$Q^\dx(x)$ on $x\ge x_i$, we compute the value $Q_{i-1}$ by solving the
nonlinear equation
\begin{equation}\label{G}
G(Q_{i-1}) = 0 \quad \mbox{where}\quad 
G(Q_{i-1}) \; \dot=\;  Q_{i-1} \;v(A(Q^\dx;x_{i-1})) - \bar f/\kappa^-, 
\end{equation}
assuming that 
\begin{equation}\label{eq:GQ}
Q_k\; v(A(Q^\dx;x_k))  = \bar f/\kappa^-,\quad \forall k \ge i.
\end{equation}

\textbf{Step 4.}
We now verify that~\eqref{G} has a unique solution. 
We compute:
\begin{eqnarray*}
G(Q_i) &=& Q_i \left[v(A(Q^\dx;x_{i-1})) -  v(A(Q^\dx;x_i)) \right] >0 ,\\
G(0) &=& -\bar f/\kappa^- <0.
\end{eqnarray*}
Furthermore, for $\dx$ sufficiently small, we have
\begin{eqnarray*}
G'(Q_{i-1}) &=& v(A(Q^\dx;x_{i-1})) + Q_{i-1} \frac{\partial v(A(Q^\dx;x_i))}{\partial {Q_{i-1}}} \\
&=& v(A(Q^\dx;x_{i-1})) + Q_{i-1}  v'(A(Q^\dx;x_i)) 
\frac{\partial A(Q^\dx;x_i)}{\partial {Q_{i-1}}} \\
&=& v(A(Q^\dx;x_{i-1})) + Q_{i-1}  v'(A(Q^\dx;x_i))  \int_{x_{i-1}}^{x_i} \frac{x_i-y}{\dx} w(y-x_{i-1})\; dy.
\end{eqnarray*}
Since 
\[\int_{x_{i-1}}^{x_i} \frac{x_i-y}{\dx} w(y-x_{i-1})\; dy=
\int_{0}^{\dx} \frac{\dx-s}{\dx} w(s)\; ds \le \frac12 w(0) \cdot \dx
= \O(1) \cdot \dx,
\]
we conclude that
\[
G'(Q_{i-1}) >0.
\]

Thus,~\eqref{G} has a unique solution of $Q_{i-1}$, satisfying $0<Q_{i-1}<Q_i$.
We remark that, numerically,~\eqref{G} can be solved efficiently by Newton iterations, 
using $Q_i$ as the initial guess.

By induction we conclude that $0<Q_{i-1}<Q_i$ for all $i<0$. 
Thus the approximate solution $Q^\dx(x)$ is positive and 
monotone increasing on $x<0$.  
By construction, it satisfies the equation~\eqref{eq:Q1}
at every grid point $x_i$. 
Taking the limit $\dx\to0$, 
$Q^\dx(x)$ converges to a limit function $Q(x)$, positive and
monotone increasing, and satisfies~\eqref{eq:Q1} for every $x$. 

\textbf{Step 5.} 
It remains to verify the limit as $x\to-\infty$.
Since $Q(x)$ is monotone and bounded, it must admit a limit 
at $x\to -\infty$.  Call it $\rho^-$.  We have
\[
\bar f = \lim_{x\to-\infty} \kappa(x) Q(x) v(A(Q;x)) = \kappa^- \rho^- v(\rho^-).
\]
We must also have $\rho^- < \hat\rho$, since by Lemma~\ref{lm:AL} 
it is the only stable asymptote as $x\to-\infty$. 
This completes the proof. 
\end{proof}

\begin{figure}[htbp]
\begin{center}
\includegraphics[width=7cm,clip,trim=5mm 5mm 10mm 10mm]{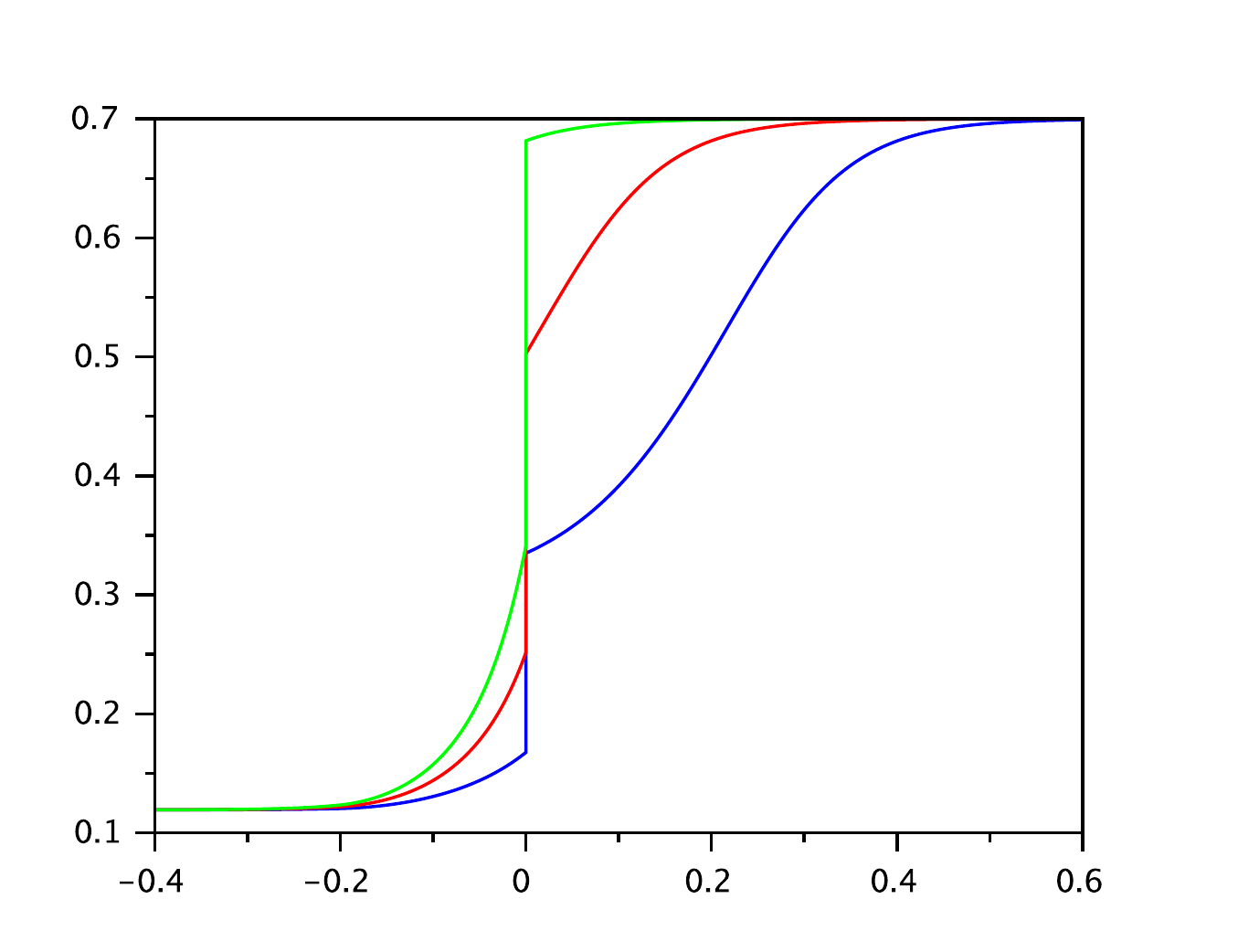}
\caption{Sample traveling waves for Case A1, with $\kappa^-=2,\kappa^+=1, h=0.2.$}
%, \rho^+=0.7$.}
\label{fig:A1}
\end{center}
\end{figure}

Several sample profiles are plotted in Figure~\ref{fig:A1}.
They are generated by numerical simulation, solving~\eqref{G} by Newton iterations.  
All profiles are bounded and monotone, continuous except  an upward jump at $x=0$. 
We further note that %in Figure~\ref{fig:A1} 
the profiles do not cross each other.
In other words, the profiles are ordered.
We have the following Corollary.

\begin{corollary}[Ordering of the profiles]\label{cor2}
Consider the setting of Theorem~\ref{th1}, and
let $Q_1(x)$ and $Q_2(x)$ be two distinct profiles.  
Then, we must have either
\[
  Q_1(x) > Q_2(x)  \quad \forall x>0 ~\mbox{and}~ x<0
\]
or 
\[
  Q_1(x) < Q_2(x)  \quad \forall x>0~ \mbox{and} ~x<0.
\]
\end{corollary}

\begin{proof}
First we observe that on $x>0$, the profiles can not cross each other, since they are
horizontal shifts of the monotone profile $W(x)$. 
Now consider two distinct profiles $Q_1,Q_2$, such that 
$Q_1(x) > Q_2(x)$ on $x> 0$.
Then $Q_1(0+) > Q_2(0+)$, 
and by the connecting condition~\eqref{eq:connect}, we have also
$Q_1(0-) > Q_2(0-)$. 

We continue by contradiction, and 
assume that there exists a point $y<0$ such that
%the graphs of $Q_1$ and $Q_2$ cross each other somewhere
%at $y$ such that 
\begin{equation}\label{QQ}
Q_1(y)=Q_2(y), \quad \mbox{and}\quad  Q_1(x)>Q_2(x) \quad \forall x>y.
\end{equation}
Then, by equation~\eqref{eq:Q1} we have
\[
Q_1(y) \kappa^- v(A(Q_1;y)) = \bar f = Q_2(y) \kappa^- v(A(Q_2;y)),
\]
which implies
\[
v(A(Q_1;y))=v(A(Q_2;y)), \qquad \mbox{i.e.} \quad  A(Q_1;y) = A(Q_2;y)
\]
a contradiction to~\eqref{QQ}. We conclude that the graphs of $Q_1,Q_2$ 
does not cross each other.
\end{proof}

\medskip

\textbf{Stability.}
We now show the stability of these profiles, 
as local attractors for solutions of the Cauchy problems for (M1) with suitable initial data. 
When $\kappa(x)\equiv 1$ is a constant function, 
the existence and uniqueness of solutions for the Cauchy problems of~\eqref{eq:CL1}  
is established in~\cite{BG2016}. % (cite) 
In particular, if the initial condition $\rho(0,x)$ is smooth in $x$, so is the solution 
$\rho(t,x)$ for $t>0$. 

However, if $\kappa(x)$ is discontinuous at $x=0$, then the solution 
$\rho(t,x)$ must have a jump at $x=0$ as well. 
A connecting condition, similar to~\eqref{eq:connect}, should be imposed:
\begin{equation}\label{eq:connect2}
\rho(t,0-) \kappa^- = \rho(t,0+) \kappa^+, \quad \forall t >0 . 
\end{equation}
We have the following definition.

%%%%%%%%%%%%%
\begin{definition}\label{def:sol1}
We say that a function $\rho(t,x)$ is a solution of~\eqref{eq:CL1} if
$\rho(t,x)$ satisfies~\eqref{eq:CL1} for all $x>0$ and $x<0$, 
and the connecting condition~\eqref{eq:connect2} at $x=0$,
for all $t> 0$. 
\end{definition}
%%%%%%%%%%%%%%

We remark that, a general existence theorem for the Cauchy problem of (M1)
with discontinuous coefficient $\kappa(x)$ is not yet available in the literature.
One may speculate that the existence of solutions could be established 
by the vanishing viscosity approach,  a possible topic for future work. 
In this paper, assuming that solutions exist for the Cauchy problem,  
we establish the local stability of the traveling wave profiles.

Let $Q^\sharp(x)$ and $Q^\flat(x)$ be two profiles such that 
$Q^\sharp(x) > Q^\flat(x)$ for every $x$. 
%be the profile with $Q(x)\equiv \rho^+$ on $x>0$,
%and let $Q^\flat(x)$ be the lower envelop for all the profiles.  
Define $D$ as the region between $Q^\sharp$ and $Q^\flat$, i.e.,
\begin{equation}\label{eq:D}
D \; \dot=\; \left\{ (x,q):  Q^\flat(x) \le q\le Q^\sharp(x) \right\}.
\end{equation}

By Corollary~\ref{cor2},  all these profiles are ordered
and they do not intersect with each other. 
One can parametrize each profile by its trace $Q(0+) $. 
For every point $(x,q)\in D$ with $x\not=0$, 
there is a unique profile that passes through it.
Fix a time $t\ge 0$. 
For each function $\rho(t,x)$ with $(x,\rho(t,x))\in D$ and $x\not=0$, 
we define the mapping:
\begin{equation}\label{eq:Phi}
\Phi(t,x) \;\dot =\; \tilde Q(0+), \quad \mbox{where $\tilde Q(x)$ is a profile s.t.~}
\tilde Q(x)=\rho(t,x).
\end{equation}

We have the following stability Theorem. 

\begin{theorem}[Stability of the profiles]\label{th3}%%%%%%%%%%%%%
Consider (M1) with $\kappa(x)$ given in~\eqref{kappa} and $\kappa^- >\kappa^+$. 
Let $\bar \rho(x)$ be the initial data, smooth except at $x=0$, satisfying
\begin{equation}\label{eq:prop0}
(x,\bar\rho(x)) \in D, \quad \forall x, 
\qquad\mbox{and}\quad
\kappa^-\bar \rho(0-)  = \kappa^+\bar \rho(0+) .
\end{equation}
Let $\rho(t,x)$ be the solution of the Cauchy problem for (M1) with initial data 
$\bar\rho(x)$,  following Definition~\ref{def:sol1}.
Then, we have
\begin{equation}\label{eq:pD}
(x,\bar\rho(t,x)) \in D, \quad \forall x, \forall t\ge 0 .
\end{equation}
Furthermore, let $\Phi(t,x)$ be the mapping for $\rho(t,x)$, 
as defined in~\eqref{eq:Phi}.  Then
\begin{equation}\label{eq:th}
\lim_{t\to\infty}  \left[\max_x\{\Phi(t,x)\} - \min_x\{\Phi(t,x)\}\right] =0. 
\end{equation}
\end{theorem}

\begin{proof} We first observe that, 
since the initial data is smooth except at $x=0$, so is the solution 
for $t>0$.  
We now claim the following:
\begin{itemize}
\item 
If $\Phi(t,y)$ is a maximum point such that
\begin{equation}\label{eq:max}
\Phi(t,y) \ge \Phi(t,x) \quad \forall x, \qquad  \Phi(t,y) > \Phi(t,x) \quad \forall x>y, 
\end{equation}
then  $\Phi_t(t,y) <0$. 
\item
Symmetrically, if $\Phi(t,y)$ is a minimum point such that
\begin{equation}\label{eq:min}
\Phi(t,y) \le \Phi(t,x) \quad \forall x, \qquad  \Phi(t,y) < \Phi(t,x) \quad \forall x>y,
\end{equation}
then  $\Phi_t(t,y) >0$. 
\end{itemize}
This claim would imply both~\eqref{eq:pD} and~\eqref{eq:th}. 

We provide a proof for the case of a maximum point, while the minimum point
can be treated in a completely similar way. 
Fix a time $t\ge 0$, and let $y$ be a point that satisfies~\eqref{eq:max}. 
We discuss 3 cases, for different locations of $y$.

\textbf{(1).} 
Consider $y>0$ which satisfies~\eqref{eq:max}, and let $Q$ be the profile such that 
\[%\begin{equation}\label{eq:c1}
Q(y) = \rho(t,y), \quad Q'(y)=\rho_x(t,y),\quad
\mbox{and}\quad
Q(x) > \rho(t,x) \quad \forall x>y.
\]%\end{equation}
Since $v'<0$, this implies
\begin{equation} \label{eq:c1a}
A(Q;y) > A(\rho;t,y)  \qquad \mbox{and}\quad 
v(A(Q;y)) < v(A(\rho;t,y)).
\end{equation}
We also have 
\begin{eqnarray*}
A(Q;y)_x &=& Q(y+h)w(h) - Q(y)w(0) - \int_0^{h} Q(y+s)w'(s)\; ds, \\
A(\rho;t,y)_x &=&  \rho(t,y+h)w(h) - \rho(t,y)w(0) - \int_0^{h} \rho(t,y+s)w'(s)\; ds.
\end{eqnarray*}
Using $w(h)=0, Q(y)=\rho(t,y), w'(s)<0$ and $Q(y+s) > \rho(t,y+s)$, 
we get
\begin{equation}\label{eq:c1ba}
A(Q;y)_x - A(\rho;t,y)_x  >0. 
\end{equation}

We now compute:
\begin{eqnarray*}
&& \hspace{-1.5cm} (\kappa^+)^{-1} \rho_t(t,y) ~=~ - \left[ \rho(t,y)  v(A(\rho; t,y))\right]_x \\
&=& \left[ Q(y)  v(A(Q; y))\right]_x- \left[ \rho(t,y)  v(A(\rho; t,y))\right]_x \\
&=& Q'(y) \left[ v(A(Q; y)) - v(A(\rho; t,y))\right] 
+ Q(y) \left[v'(A(Q;y)) - v'(A(\rho;t,y)) \right] A(Q;y)_x \\
&&+~ Q(y) v'(A(\rho;t,y)) \left[A(Q;y)_x - A(\rho;t,y)_x \right] \\
&<&0.
\end{eqnarray*}
Here the last inequality holds  thanks to~\eqref{eq:c1a}-\eqref{eq:c1ba}.

\medskip

\textbf{(2).} Consider $y=0$.   
By the connecting conditions~\eqref{eq:connect} and~\eqref{eq:connect2}, we have
\[
Q(0-) = \rho(t,0-),  \qquad
Q(0+) = \rho(t,0+),
\]
Since the derivative $Q'(x)$ in~\eqref{eq:Q2} is the left limit,  we consider  $y=0-$. 
We have
\[%\begin{equation}\label{eq:cc1}
Q(0-) = \rho(t,0-),\qquad Q'(0-) \le  \rho_x(t,0-), \qquad 
Q(x) > \rho(t,x) ~\forall x>0.
\]%\end{equation}
We have similar estimates as those in~\eqref{eq:c1a}-\eqref{eq:c1ba}, i.e.,
\begin{equation} \label{eq:cc1a}
A(Q;0-) > A(\rho;t,0-), \quad 
v(A(Q;0-)) < v(A(\rho;t,0-)), \quad
A(Q;0-)_x > A(\rho;t,0-)_x.
\end{equation}

We now compute, by using ~\eqref{eq:cc1a}
\begin{eqnarray*}
&& \hspace{-1.5cm} (\kappa^-)^{-1}\rho_t(t,0-) 
~=~ \left[ Q(0-)  v(A(Q; 0-))\right]_x- \left[ \rho(t,0-)  v(A(\rho; t,0-))\right]_x \\
&= & Q'(0-) \left[ v(A(Q; 0-)) - v(A(\rho; t,0-))\right] + v(A(\rho;t,0-) [Q'(0-)-\rho_x(t,0-)]\\
&& + Q(0-) \left[v'(A(Q;0-)) - v'(A(\rho;t,0-)) \right] A(Q;0-)_x \\
&&+ Q(0-) v'(A(\rho;t,0-)) \left[A(Q;0-)_x - A(\rho;t,0-)_x \right] \\
&<&0.
\end{eqnarray*}

\medskip

\textbf{(3).} 
If $y<0$, then  the averaging operator could possibly be taken over the jump location at
$x=0$.  Nevertheless, we still have 
\[
Q(y) = \rho(t,y), 
\quad Q'(y) = \rho_x(t,y),
\quad Q(x) > \rho(t,x) \quad \forall x>y.
\]
The rest of the computation remains the same as in the previous steps. 
\end{proof}

We perform a numerical simulation with Riemann initial data $(\rho^-,\rho^+)$,
and the plots are shown in Figure~\ref{fig:LFA1}. 
We observe that, the solution approaches a traveling wave profile as $t$ grows. 
Note that the Riemann initial data actually do not satisfy the assumptions in Theorem~\ref{th3},
indicating that Theorem~\ref{th3} probably applies to more general functions of initial data.

\begin{figure}
\begin{center}
\includegraphics[width=14cm]{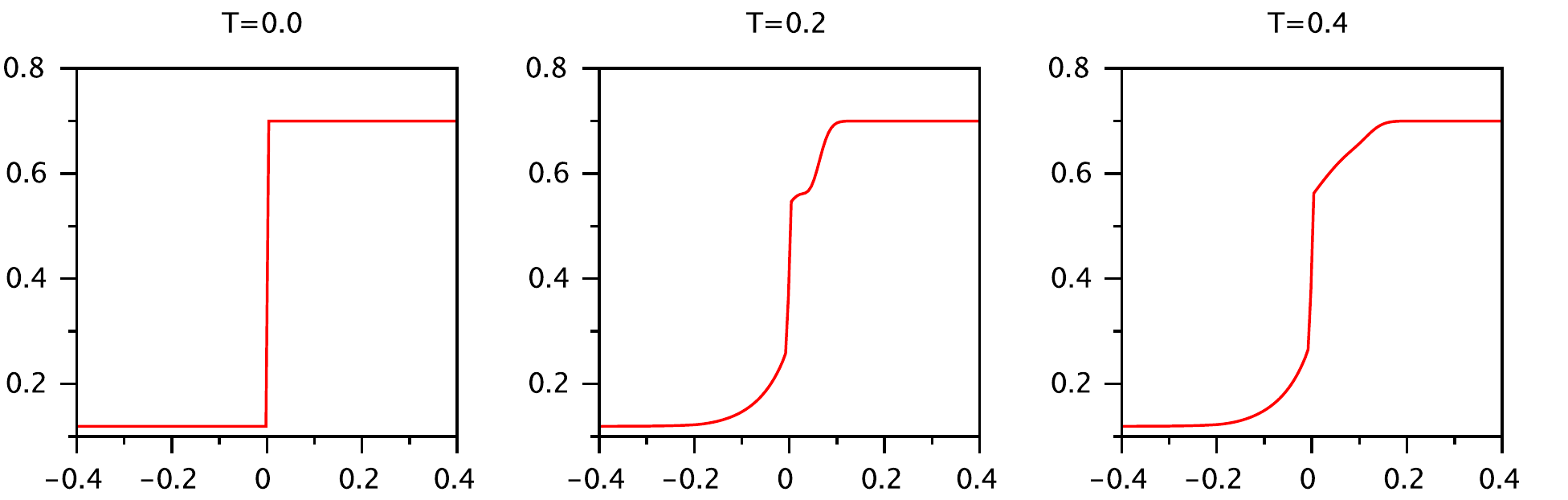}
\caption{Numerical simulation for model (M1) with Riemann initial data for Case A1.}
\label{fig:LFA1}
\end{center}
\end{figure}

%%%%%%%%%%%
\subsubsection{Case A2:  $ \rho^- < \rho^+\le \hat\rho$}
\label{sec:A2}

Since $\rho^+ \le \hat\rho$ is an unstable asymptote for $x\to\infty$, 
the solution on $x>0$ must be $Q(x)\equiv \rho^+$. 
At $x=0$  the connecting condition~\eqref{eq:connect} implies $Q(0-) < Q(0+)$, 
therefore the profile has an upward jump. 
Finally, the solution can be extended to $x<0$ by solving an initial value problem 
for~\eqref{eq:Q1}, with initial data given on $x>0$, and a jump at $x=0$. 
Using the same argument as in Theorem~\ref{th1} for Case~A1, we conclude that
this unique profile is monotone increasing on $x<0$. 
We have the following Theorem.

\begin{theorem}[Existence of a unique profile]\label{th2}%%%%%%%%%%%%%
Let $\kappa^->\kappa^+$, and 
let $\rho^-,\rho^+$ be given 
such that $\rho^- < \rho^+ \le\hat\rho$
and~\eqref{eq:bf} holds. 
Then, there exists exactly one  stationary monotone profile $Q(x)$, 
monotone increasing on $x<0$, satisfying
\[
Q(x)\equiv \rho^+ \quad (x>0), \qquad \mbox{and} \quad 
\lim_{x\to-\infty} Q(x) = \rho^-.
\]
\end{theorem}

See Figure~\ref{fig:A2} for a sample profile.

\begin{figure}[htbp]
\begin{center}
\includegraphics[width=7cm,height=4cm,clip,trim=2mm 1mm 9mm 5mm]{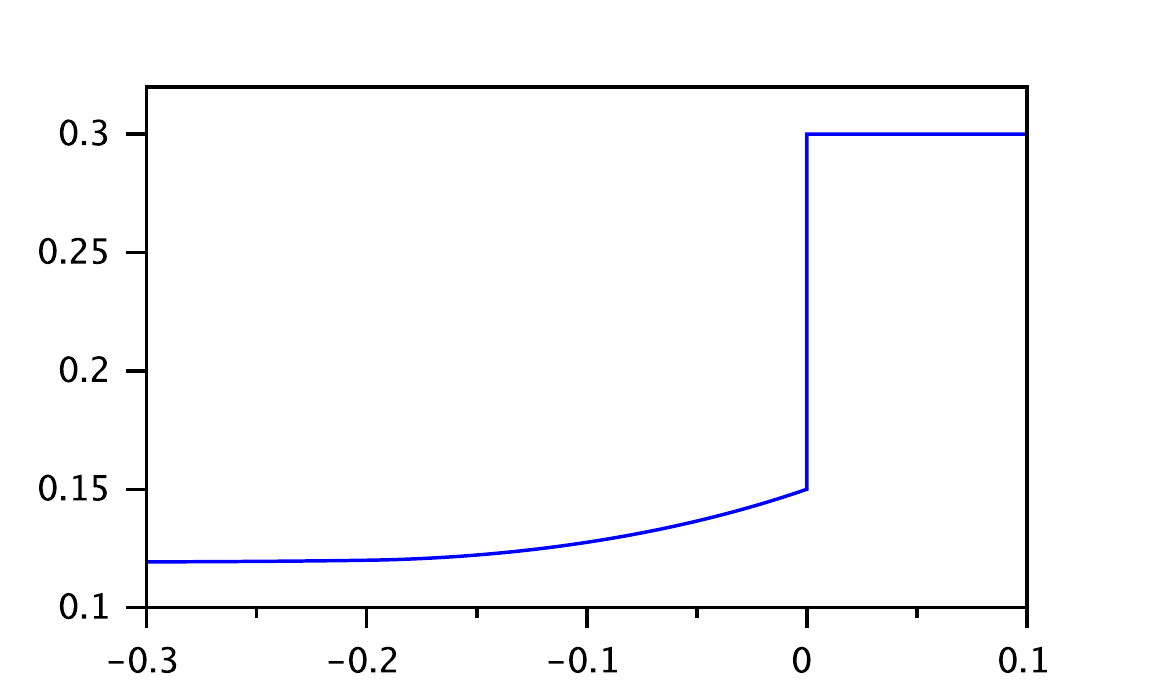}
\caption{Typical traveling wave profile for Case A2}
\label{fig:A2}
 
 \bigskip
 
\includegraphics[width=13cm,clip,trim=0mm 1mm 5mm 2mm]{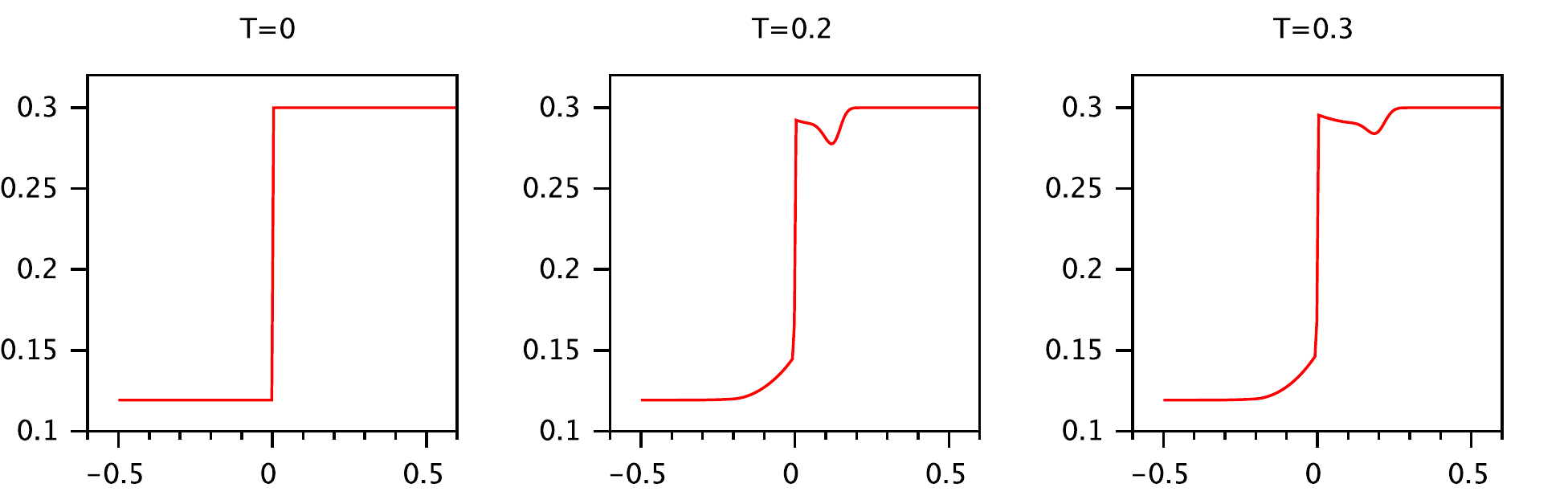}
\caption{Numerical simulation for the PDE model with Riemann initial data for Case A2.}
\label{fig:LFA2}

\bigskip

\includegraphics[width=13cm,clip,trim=0mm 1mm 5mm 2mm]{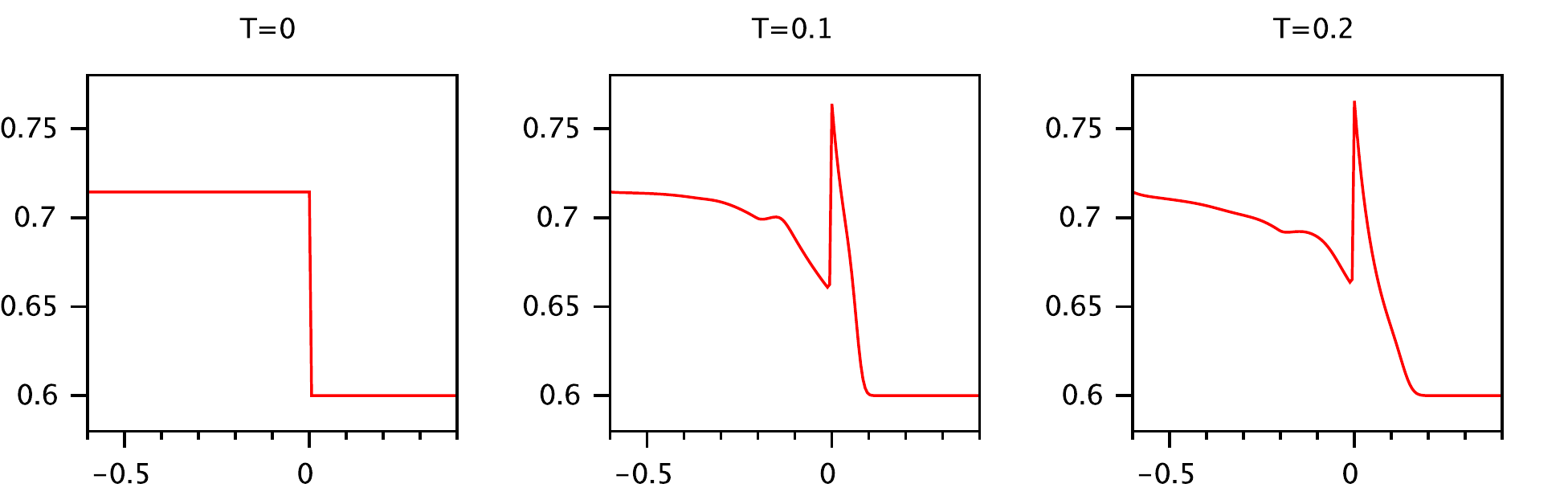}
\caption{Numerical simulation for the PDE model with Riemann initial data for Case A3.}
\label{fig:LFA3}

\bigskip

\includegraphics[width=13cm,clip,trim=0mm 1mm 5mm 2mm]{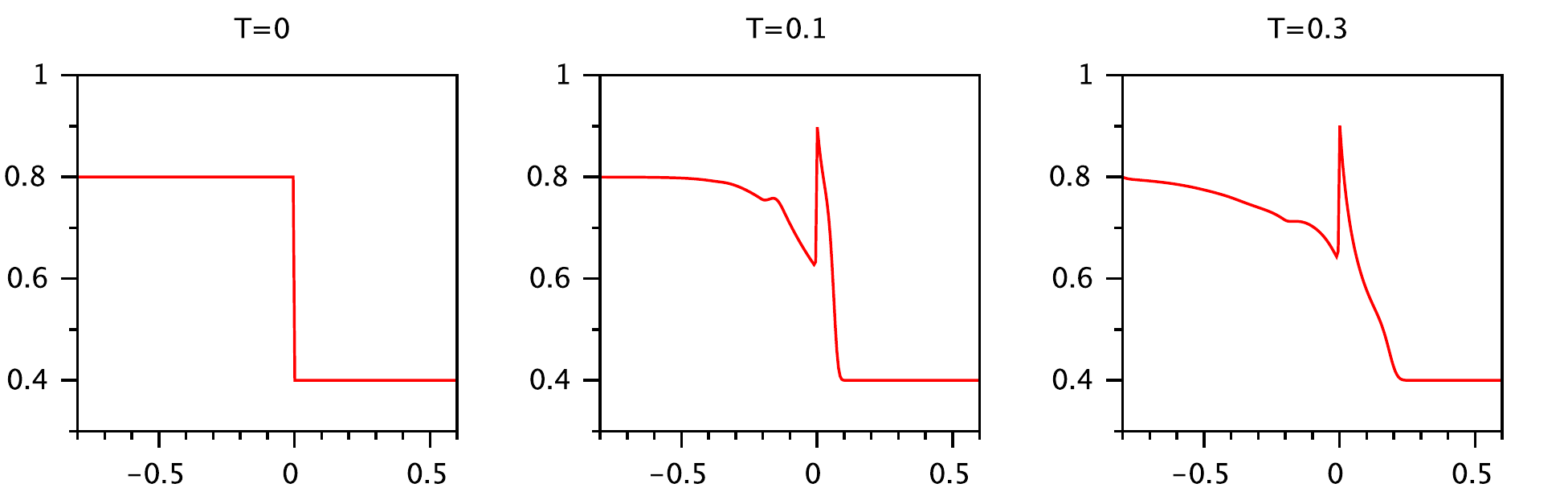}
\caption{Numerical simulation for the PDE model with Riemann initial data for Case A4.}
\label{fig:LFA4}
\end{center}
\end{figure}

%\paragraph{Stability.} 
Since $\rho^+\le\hat\rho$  is an unstable asymptote for $x\to \infty$,
the profile is not a local attractor for solutions of the Cauchy problems for (M1). 
In fact, any perturbation that enters the region $x>0$ will persist,
as verified  by a numerical simulation  in Figure~\ref{fig:LFA2}. 
We observe that, with Riemann initial data, 
an oscillation is formed around the origin and then travels into the region $x>0$,
where it travels further to the right as $t$ grows.
%Since $\rho^+$ is not a stable asymptote for $x\to +\infty$,
%the oscillation persists as $t$ grows.

%%%%%%%%%%%
\subsubsection{Case A3:  $\hat\rho\le \rho^+<\rho^-$}
\label{sec:A3}

Since $\rho^- > \hat \rho$,
the asymptote $\rho^-$ is unstable as $x\to-\infty$. 
If a profile shall exists, it must be constant $\rho^-$ on $x<0$.
There exists no profile on $x>0$ that can be connected to this constant solution.
In conclusion, no stationary profiles exist for this case.

A numerical simulation with Riemann initial data is 
performed, and results are plotted in Figure~\ref{fig:LFA3}. 
We observe that oscillations are formed around $x=0$, 
which travel to the left in the region $x<0$,
where they persist as $t$ grows.

%%%%%%%%%%%
\subsubsection{Case A4:  $ \rho^+ < \hat\rho < \rho^-$}
\label{sec:A4}

Since both $\rho^-$ and $\rho^+$ are unstable, there does not exist any profiles. 
We present a numerical simulation with Riemann data, and plot the results 
in Figure~\ref{fig:LFA4}.
We observe rather wild behavior. 
Oscillations are formed around $x=0$, and then 
enter both regions of $x>0$ and $x<0$, and they persist as $t$ grows.

%%%%%%%%
\subsection{Case B:  $\kappa^- < \kappa^+$}
\label{sec:B}

Fix $\bar f > 0 $ which lies in the ranges of $f^-$ and $f^+$, 
there exist unique values $\rho_{1},\rho_2,\rho_3,\rho_4$ such that
\[%\begin{equation}\label{B:fs}
\rho_1 < \rho_2 \le \hat\rho\le \rho_3 <\rho_4,\qquad
f^+(\rho_1)=f^+(\rho_4) =\bar f = f^-(\rho_2)=f^-(\rho_3).
\]%\end{equation}
Similar to Case A, there are 4 sub-cases: 
\begin{itemize}
\item[B1.] We have $\rho^- =\rho_2, \rho^+=\rho_4$ and $\rho^-<\hat\rho<\rho^+$; 
\item[B2.] We have $\rho^- =\rho_1, \rho^+=\rho_2$ and $\rho^+<\rho^-\le\hat\rho$; 
\item[B3.] We have $\rho^- =\rho_3, \rho^+=\rho_4$ and $\hat\rho\le \rho^-<\rho^+$; 
\item[B4.] We have $\rho^- =\rho_1, \rho^+=\rho_3$ and $\rho^+<\hat\rho<\rho^-$.
\end{itemize}

Note that
since $\kappa^- < \kappa^+$,  the connecting condition~\eqref{eq:connect}
implies that $Q(0-)>Q(0+)$, thus $Q(x)$ has a downward jump at $x=0$.
This means that the profiles are no longer monotone increasing. 
Furthermore, since we seek profiles with $Q(0-)\le 1$, 
this imposes a restriction  on the trace $Q(0+)$,
\begin{equation}\label{eq:Q+c}
Q(0+) \le \frac{\kappa^-}{\kappa^+}. %<1 
\end{equation}
%(The actual constraint would be stricter, for a profile to exist.) 

As previously, we let $W(x)$ % (or any horizontally shifted version) 
denote a stationary profile where $\kappa(x)\equiv\kappa^+$, discussed in
Section~\ref{sec:review}.
Below, we discuss each sub-case in detail. 
We remark that the overall framework of the discussions  is similar to that of Case A,
but details (especially for Case B1) are rather different, due to the lack of monotonicity.

%%%%%%%%
\subsubsection{Case B1:  $\rho^-<\hat\rho<\rho^+$}
\label{sec:B1}

%We neglect the trivial case $\rho^-=0,\rho^+=1$ and consider  
%$0<\rho^- <\hat\rho<\rho^+<1$. 

\begin{theorem}\label{tm:B1}
Given $\kappa^- <\kappa^+$, and $\rho^-,\rho^+$ satisfying
\[
\rho^- < \hat\rho <\rho^+, \qquad f^-(\rho^-) = f^+(\rho^+) >0,
\]
there exist infinitely many solutions $Q(x)$ 
to~\eqref{eq:Q3} which satisfy the ``boundary conditions''
\[
\lim_{x\to-\infty} Q(x) = \rho^-, \quad \lim_{x\to+\infty} Q(x) = \rho^+.
\]
All profiles are monotone increasing on $x>0$, but some profiles may be oscillatory on $x<0$.
\end{theorem}

\begin{proof} 
The proof takes several steps.

\textbf{(1)} 
On $x>0$, $Q(x)$ could  be either the constant function $Q(x)\equiv \rho^+$
or some horizontal shift of $W(x)$. 
We first claim that the constant solution 
$Q(x)\equiv \rho^+$  on $x>0$ will not give a profile on $x<0$. 
Indeed, by  the connecting condition~\eqref{eq:connect} we have $Q(0-) > Q(0+)=\rho^+$. 
Then one can easily show that $Q'(0-) <0$, and furthermore 
$Q'(x)<0$ for $x<0$. Thus, $Q(x)$ reaches 1 for some finite $x<0$, and the solution 
can not be continued further as $x$ becomes smaller.

In conclusion, on $x>0$, $Q(x)$ is some horizontal shift of $W(x)$. 
At $x=0$, the trace $Q(0+)$ takes value between $\rho_1$ and $\rho^+$, 
where $\rho_1<\rho^-< \hat\rho<\rho^+$ and 
$f^+(\rho_1)=f^+(\rho^+)\;\dot=\;\bar f>0$.  

\medskip

\textbf{(2)}  
Furthermore, we consider only the profiles where $Q(0+)$ satisfies~\eqref{eq:Q+c}. 
At $x=0$, the connecting condition~\eqref{eq:connect} applies, 
which determines the value for the trace $Q(0-)$.
We then solve an initial value problem for $Q(x)$ on $x<0$, 
with initial data given on $x\ge 0$.  
We expect to have  infinitely many profiles. 
From the same argument as in Corollary~\ref{cor2}, 
all profiles are ordered and will never cross each other. 

\medskip

\textbf{(3)}   Utilizing the same analysis as in Step 3-4 of Theorem~\ref{th1},
we have that, if a profile $Q(x)$ is monotone on $[-h,0]$, then
it remains monotone for $x<-h$. 

\medskip

\textbf{(4)}  We now construct a lower envelope $Q^\flat$ for all the profiles on $x<0$,
by solving the initial value problem with initial data 
$Q(x)\equiv \rho_1< \rho^-$ for $x>0$.
Note that any horizontal shift of the profile $W(x)$ 
%(and any of its horizontally shifted  version)
satisfies $W(x) > \rho_1$ for $x>0$. 
Since the profiles can not cross each other, we conclude that, 
any profiles with $Q(x)=W(x)$ on $x>0$  will lie above $Q^\flat$. 

We claim that, on $x<0$ the profile  $Q^\flat$ is monotone decreasing 
and it lies below $\rho^-$. 
Indeed, since $\kappa^+\rho_1 v(\rho_1) = \kappa^- \rho^- v(\rho^-)=\bar f$,
the connecting condition~\eqref{eq:connect}  gives
\[Q(0-) = \frac{\kappa^+Q(0+)}{\kappa^-}
= \frac{\kappa^+\rho_1}{\kappa^-} =
\frac{\kappa^- \rho^- v(\rho^-)}{v(\rho_1) \kappa^- }
= \frac{ \rho^- v(\rho^-)}{v(\rho_1)  } < \rho^-.\]
Then, by~\eqref{eq:Q2},  we have
\[ (Q^\flat)'(0-)     v(A(Q^\flat;0)) = -Q(0-)  v'(A(Q^\flat;0)) A(Q^\flat;0-)_x <0,\]
where the last inequality holds thanks to $Q(0-) > Q^\flat(h)=\rho_1$ 
and we have $A(Q^\flat;0-)_x <0$.
A contradiction  argument  shows  that $(Q^\flat)'(x) <0$  for every $x<0$. 
Thus, $Q^\flat$ is monotone decreasing on $x<0$. 

To show that $Q^\flat(x) < \rho^-$ on $x<0$, we proceed with contradiction. 
We assume that there exists a $y<0$ 
such that $Q^\flat(y)=\rho^-$ and $Q^\flat(x) < \rho^-$ for $x>y$. 
By~\eqref{eq:Q1} we compute
\[
Q^\flat(y) \kappa^- v(A(Q^\flat;y)) = \bar f = \kappa^- \rho^- v(\rho^-)
\]
thus
\[
v(A(Q^\flat;y)) = v(\rho^-) \quad \rightarrow \quad A(Q^\flat;y) = \rho^-,
\]
a contradiction which proofs our claim.

Finally, by the same argument as in Step 5 in the proof of Theorem~\ref{th1}
one concludes that the profile $Q^\flat(x)$ approaches $\rho^-$ asymptotically
as $x\to -\infty$. 

\medskip

\textbf{(5)} 
By the ordering of the profiles,  any profile $Q(x)$ with $Q(0-)>\rho_1$, if it exists,
would lie above $Q^\flat(x)$. 
However $Q(x)$ might lose monotonicity and become oscillatory around $\rho^-$
on $x<0$.
If this happens, we claim  that the local max of an oscillating solution 
is decreasing and  approaches $\rho^-$ as $x\to -\infty$. 

Indeed, by~\eqref{eq:Q1},  on $x<0$ we have  
\[ 
Q'(x) = Q(x) \frac{-v'(A(Q;x))}{v(A(Q;x))} A(Q;x)_x.
\]
Thus, $Q'(x)$ has the same sign as $A(Q;x)_x$. 
By the assumption $w(h)=0$, we  have
$w(0) = -\int_0^h w'(s)\; ds$.
Thus, we compute, for $x<0$
%\[%\begin{eqnarray*}
%A(Q;x)_x %= \int_0^h Q'(x+s) w(s) \; ds 
%= Q(x+h) w(h) - Q(x) w(0) - \int_0^{h} Q(x+s) w'(s) \; ds.
%\]%\end{eqnarray*}
%
%Therefore we can write
\[
A(Q;x)_x = \int_0^{h} [Q(x+s) - Q(x)] (-w'(s)) \; ds. 
\]
Since $w'<0$, we conclude that, if $A(Q;x')_x =0$ for some $x'$, then 
on $[x',x'+h]$ we have either $Q(x)\equiv Q(x')$ or $Q(x)$ oscillated around $Q(x')$.
In particular, this implies that if $y<-h$ is a local max with $Q(y) > \rho^-$, 
such that $A(Q;y)_x =0$, 
then we must have $Q(x^*) > Q(y)$ for some $x^*\in (y,y+h)$.
Therefore, there exists a local max on the left of $y$, with a higher max value. 

Thus, there exists a sequence of local maxima $y_k$ %($k=1,2,\cdots$) 
with 
\[ y_{k+1} < y_k, \qquad Q(y_k) >\rho^-, \qquad Q(y_{k+1})  < Q(y_k)\qquad 
\forall k. \]
The sequence might be finite or infinite. 
We conclude that there exists an increasing function on $x<0$ which serves as an 
upper envelope for this oscillatory profile. 
Since the  flux equals $\bar f$, this envelope  approaches $\rho^-$ asymptotically 
as $x\to -\infty$. 

%\medskip

By continuity there exists an upper profile $Q^\sharp(x)$, whose graph lies
between $Q^\flat(x)$ and the upper envelope. 
In particular, $Q^\sharp(x)$ approaches the limit $\rho^-$ as $x\to-\infty$.
%(Although, the exact location of this upper envelope is hard to find.)

We conclude that, 
in between the profiles $Q^\flat(x)$ and $Q^\sharp(x)$,  
there exist infinitely many profiles for $Q(x)$. 
These profiles never cross each other, and they all approach $\rho^-$ as $x \to -\infty$.
\end{proof}

Sample profiles for Case B1 are given in Figure~\ref{fig:B1}. 
By a similar argument as for Case A1, one concludes that these profiles are time asymptotic
limits for solutions of the Cauchy problems for (M1).  We omit the details of the proof. 
Result of a numerical simulation  is given in Figure~\ref{fig:LFB1},
with Riemann initial data.
We observe that the solution approaches a certain traveling wave profile as $t$ grows.

\begin{figure}[htbp]
\begin{center}
\includegraphics[width=7.5cm,clip,trim=5mm 1mm 12mm 6mm]{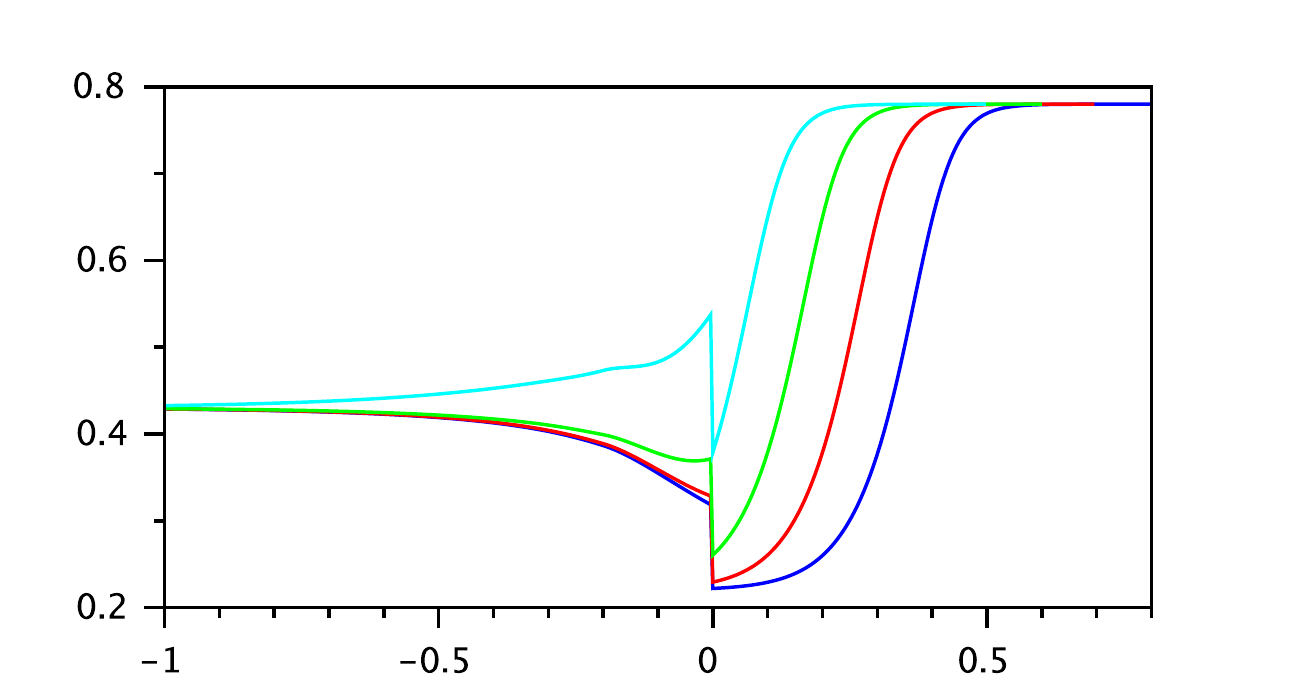}
\caption{Sample traveling waves for Case B1. }
\label{fig:B1}

\bigskip

\includegraphics[width=13cm,clip,trim=0mm 1mm 5mm 1mm]{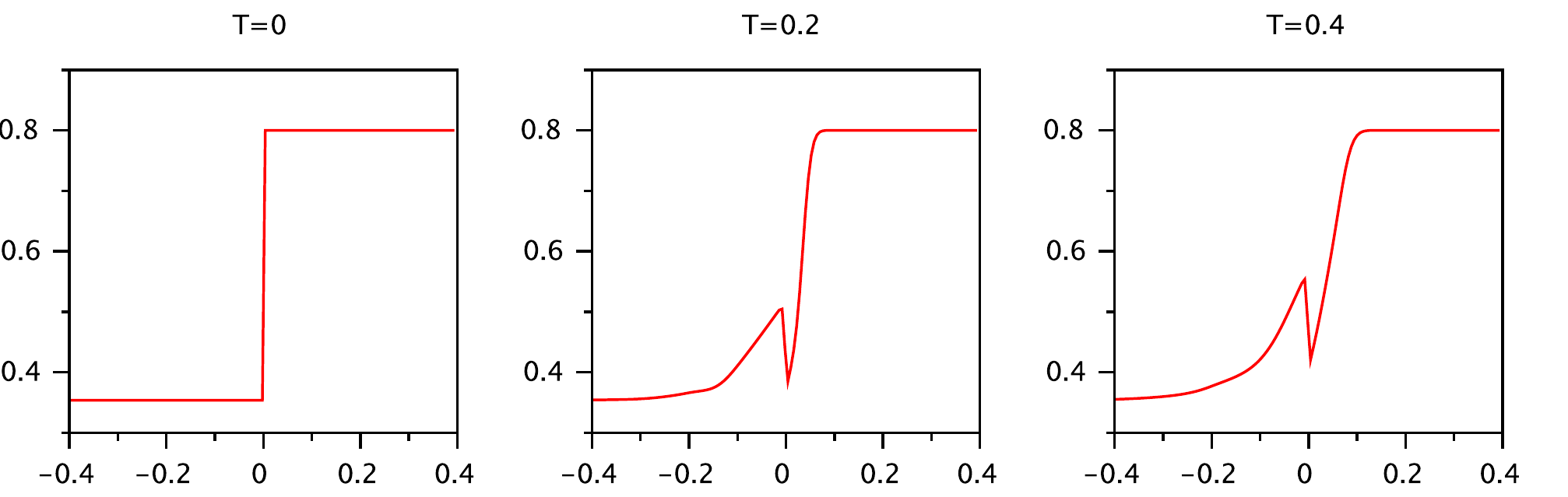}
\caption{Numerical simulation for the PDE model with Riemann initial data for Case B1.}
\label{fig:LFB1}
\end{center}
\end{figure}

%%%%%%%%
\subsubsection{Case B2, Case B3, and Case B4}
\label{sec:B234}

For Case B2, we have $\rho^+<\rho^-\le\hat\rho $, which 
is the counter part for Case A2. 
Since $\rho^+ \le\hat\rho$ is an unstable asymptote as $x\to\infty$, 
we must have $Q(x) \equiv \rho^+$ on $x>0$. 
At $x=0$ we apply the connecting condition~\eqref{eq:connect} to get $Q(0-)$. 
%$Q(0-)=Q(0+)\kappa^+/\kappa^-$.
We then solve an initial value problem on $x<0$.
By the same arguments as for Case A2
we  prove the existence of a monotone decreasing profile on $x<0$.
A sample profile is illustrated in Figure~\ref{fig:B2}.
Unfortunately, such an profile is not a local attractor for the solutions
of the Cauchy problem for (M1).
Result of a numerical simulation is given in Figure~\ref{fig:LFB2}, 
with Riemann initial data. 
We observe that an oscillation is formed around $x=0$, which travels into
the region $x>0$, where it persists as $t$ grows.

For case B3, we have $\hat\rho\le \rho^-<\rho^+$
and for Case B4 we have $ \rho^+<\hat\rho<\rho^-$. 
These are the counter parts for Case A3 and Case A4 respectively, 
and there are no profiles. 
We plot the results of a numerical simulation in Figure~\ref{fig:LFB3} for Case B3
and in Figure~\ref{fig:LFB4} for Case B4,
and observe the oscillations in solutions which persist in time.

\begin{figure}[htbp]
\begin{center}
\includegraphics[width=7cm,clip,trim=5mm 1mm 12mm 6mm]{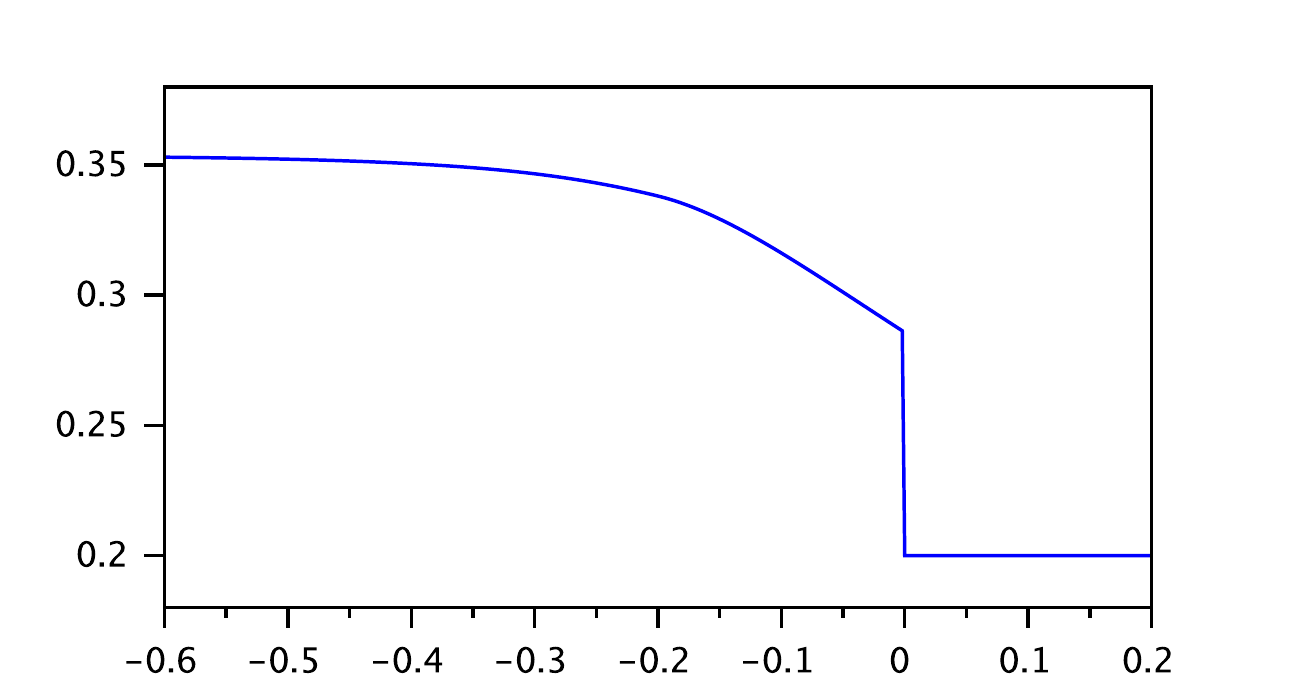}
\caption{Sample traveling wave for Case B2.}
\label{fig:B2}

\bigskip

\includegraphics[width=13cm,clip,trim=0mm 1mm 5mm 1mm]{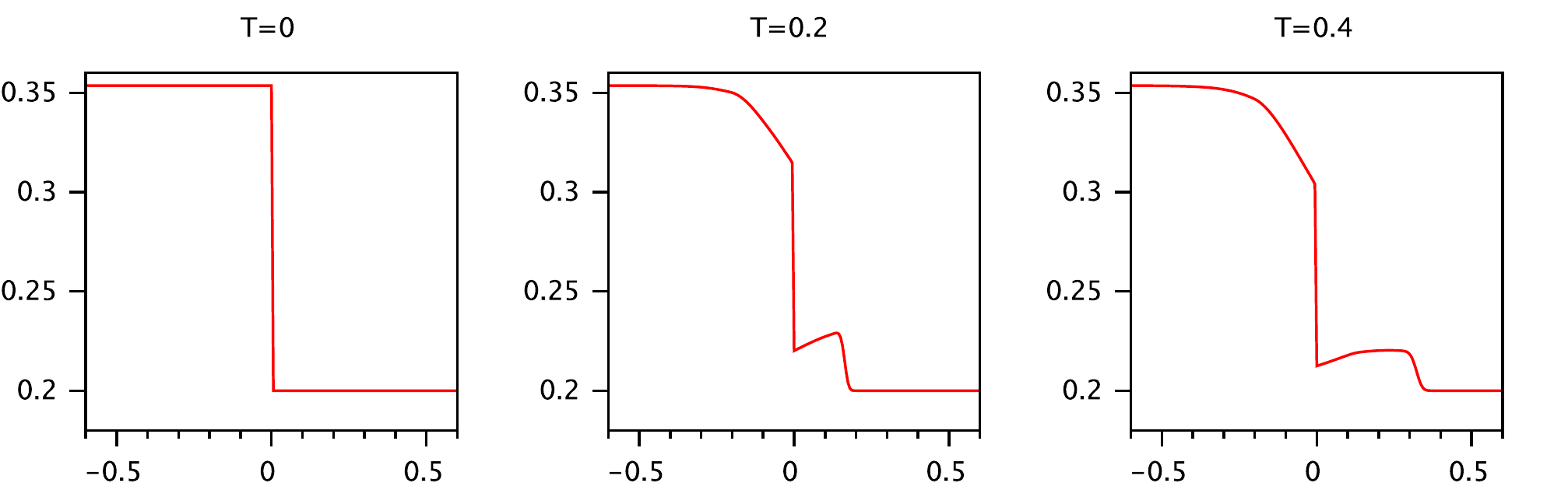}
\caption{Numerical simulation for the PDE model with Riemann initial data for Case B2.}
\label{fig:LFB2}

\bigskip

\includegraphics[width=13cm,clip,trim=0mm 1mm 5mm 1mm]{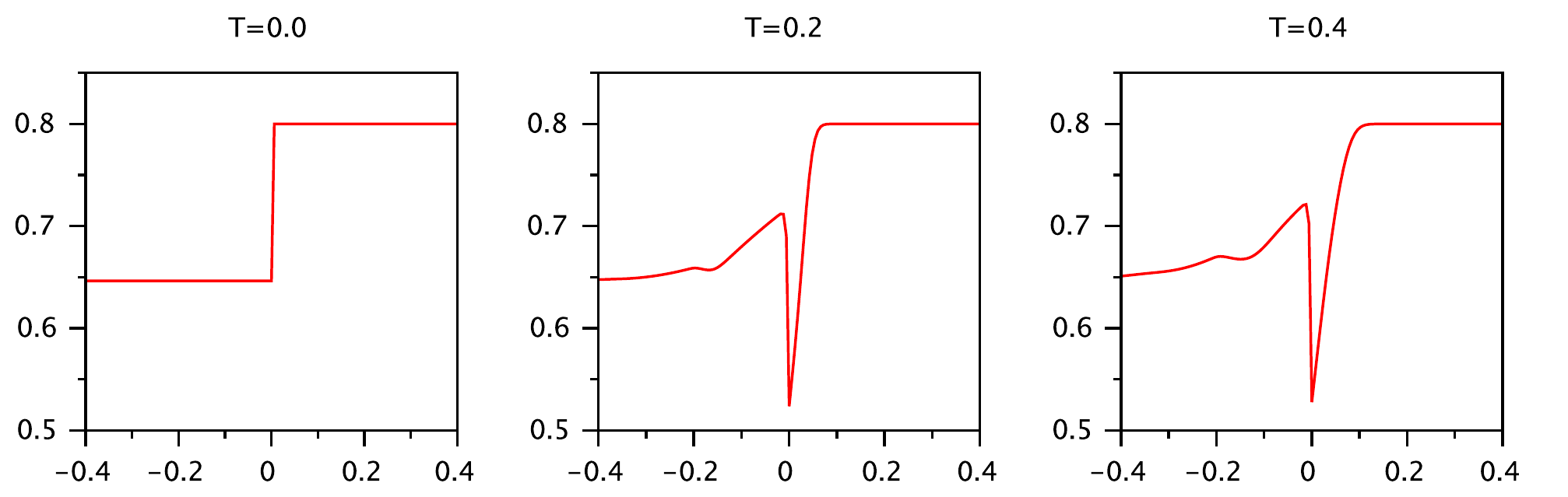}
\caption{Numerical simulation for the PDE model with Riemann initial data for Case B3.}
\label{fig:LFB3}

\bigskip

\includegraphics[width=13cm,clip,trim=0mm 1mm 5mm 1mm]{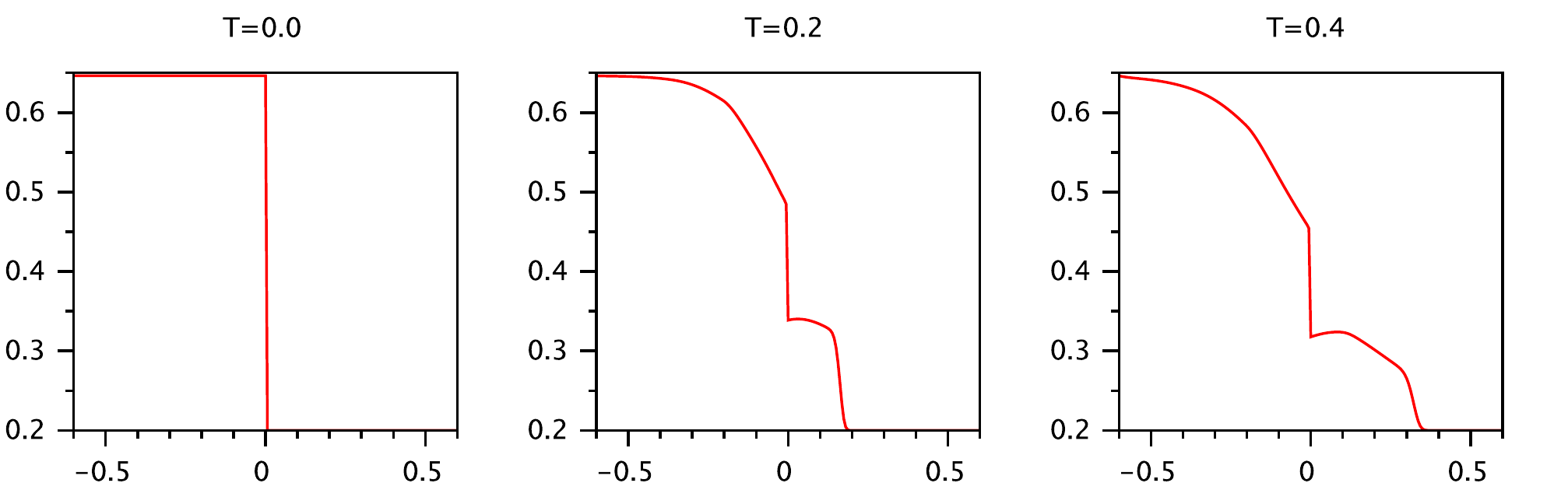}
\caption{Numerical simulation for the PDE model with Riemann initial data for Case B4.}
\label{fig:LFB4}
\end{center}
\end{figure}

%%%%%%%%%%%%%%%%%%%%%%%%%%%
\section{Stationary traveling wave profiles for (M2)}\label{sec:M2}
\setcounter{equation}{0}

Let $P(x)$ be a stationary profile for (M2). 
It satisfies the equation
\begin{equation}\label{C3-0}
P(x) \cdot \int_x^{x+h} \kappa(y) v( P(y)) w(y-x) \; dy \equiv \bar f ,
%\qquad \mbox{where}\quad \bar f= f^-(\rho^-)=f^+(\rho^+).
\end{equation}
where 
\[
\bar f = \lim_{x\to\pm\infty} P(x) \cdot \int_x^{x+h} \kappa(y) v( P(y)) w(y-x) \; dy.
\]
In the case where $\lim_{x\to\pm\infty} P(x) = \rho^\pm$ and $\kappa(x)$ 
satisfies~\eqref{kappa}, we have 
\[
\bar f= f^-(\rho^-)=f^+(\rho^+). 
\]
%Recall the definition of $V_2(x)$ in~\eqref{V2}, i.e., %
Denote now 
\[
V(x) \; \dot=\;\int_x^{x+h} \kappa(y) v( P(y)) w(y-x) \; dy ,
%~=~ \int_0^h \kappa(x+s) v(P(x+s)) w(s) \; ds,
\]
\eqref{C3-0} can written as
\begin{equation}\label{P}
P(x) V(x) \equiv \bar f \qquad \implies \quad P'(x) V(x) + P(x) V'(x)=0.
\end{equation}

We see that 
$V(x)$ is Lipschitz continuous even with discontinuous $P(x)$. 
This implies that $P(x)$ is also Lipschitz continuous, but with a kink at $x=0$.
We compute
\begin{eqnarray*}
V'(x) 
&=& \int_0^h \left[ \kappa(x+s) v(P(x+s))_x + (\kappa^+-\kappa^-)\delta_0(x+s) v(P(x+s))  \right] w(s) \; ds,
\end{eqnarray*}
where $\delta_0(x)$ denote the Dirac delta function with unit mass concentrated at $x=0$.
Writing it out with piecewise details, we have
\begin{equation}\label{V}
V'(x) = \begin{cases}
\displaystyle \kappa^+ \cdot \int_0^h   v(P(x+s))_x  w(s) \; ds & (x>0) \\[2mm]
\displaystyle  \kappa^-  \int_0^{-x}  v(P(x+s))_x  w(s) \; ds 
+ (\kappa^+-\kappa^-) v(P(0)) w(-x)  \\
\displaystyle \qquad  + ~\kappa^+  \int_{-x}^h   v(P(x+s))_x  w(s)ds & (-h <x<0)\\[2mm]
\displaystyle \kappa^- \cdot \int_0^h   v(P(x+s))_x  w(s) \; ds & (x<-h) 
\end{cases} 
\end{equation}
Since $w(h)=0$, then $V'(x)$ is continuous at $x=-h$. 
At $x=0$,  $V'(x)$ has a jump
\[
V'(0+) - V'(0-) = (\kappa^- - \kappa^+) v(P(0)) w(0).
\]

In summary, we seek stationary profiles $P(x)$ which satisfies the equation
\begin{equation} \label{C3} 
P(x) \cdot \int_x^{x+h} v( P(y)) w(y-x) \; dy = \begin{cases} \bar f/\kappa+ , & (x>0)\\
\bar f/\kappa^-, & (x<0) \end{cases}, 
\end{equation}
and the connecting condition at $x=0$:
\begin{equation}\label{connectC}
P'(0+)-P'(0-) = (\kappa^+-\kappa^-)  \frac{P(0)}{V(0)} v(P(0)) w(0).
\end{equation}

%\medskip

\paragraph{Previous results.} 
Consider the simpler case where $\kappa(x)\equiv \bar\kappa$ is a constant function,
% say $\kappa(x)\equiv \bar\kappa$, 
and let $\cW(x)$ be a stationary profile for (M2). 
Then, $\cW(x)$ satisfies the integral equation
\begin{equation}\label{cW}
\cW(x) \cdot \int_x^{x+h} v( \cW(y)) w(y-x) \; dy = \frac{\bar f }{\bar\kappa}.
\end{equation} 
By the results in~\cite{RidderShen2018}, Lemma~\ref{lm:AL}  and Theorem~\ref{th:O1}
in Section~\ref{sec:review} hold  for the profile $\cW(x)$.

\medskip

Below we  consider two cases, where Case C is for $\kappa^->\kappa^+$,
and Case D for $\kappa^-<\kappa^+$. 
We assume that $\bar f>0$, since the cases with $\bar f=0$ are trivial.

%%%%%%%%%%%%%%%%%%%%%%%%%%
\subsection{Case C: $\kappa^->\kappa^+$}
Similar to Case A, we let $\rho_{1},\rho_2,\rho_3,\rho_4$ 
be the unique values that satisfies~\eqref{eq:cA}. 
Then, we have 4 sub-cases of C1, C2, C3, C4, which are the corresponding cases 
of A1, A2, A3, A4, respectively. 
We also denote by $\cW(x)$ the stationary profile with $\kappa(x)\equiv\kappa^+$ and
$\lim_{x\to\infty} \cW(x) = \rho^+$.

%%%%%%%%%%%%%%%%%%%%%%%%%%%%%%
\subsubsection{Case C1:  $\rho^- < \hat \rho < \rho^+$}

%Again we neglect the trivial case $\rho^-=0,\rho^+=1$, and consider 
%$0<\rho^-<\hat\rho<\rho^+<1$. 
%Here  $\rho^-$ and $\rho^+$ are stable asymptotes for
%$x\to -\infty$ and $x\to\infty$, respectively. 
%We have the following existence Theorem.

\begin{theorem}\label{th:C1}
Let $\kappa^->\kappa^+$ and let $\rho^-,\rho^+$ be given such that 
$0<\rho^-< \hat\rho < \rho^+<1$, and assume that~\eqref{eq:ff} holds. 
Then, there exist infinitely many stationary monotone profiles $P(x)$ 
which satisfy the equation~\eqref{C3} and the boundary conditions~\eqref{eq:BC}.
At $x=0$ the profiles $P(x)$ are continuous but not differentiable. 
\end{theorem}

\begin{proof}  The proof takes several steps.

\textbf{Step 1.} 
Since $\rho^+>\hat\rho$ is a stable asymptote for $x\to\infty$, then
on $x\ge 0$ the stationary profile can be either the constant function 
$P(x)\equiv \rho^+$, or some horizontal shift of the profile $\cW(x)$. 
In all cases, the profile is smooth and monotone. 
With this ``initial data'' given, the profile $P(x)$ on $x<0$ can be obtained by solving
an initial value problem of~\eqref{C3} backward in $x$.

\medskip

\textbf{Step 2.} 
We now construct an approximate solution to the initial value problem 
in a similar way as in the proof of Theorem~\ref{th1} (Step 3). 
We discretize the region $x<0$ with a uniform mesh with size $\Delta x$,
and denote the grid point by $x_i = i \Delta x$, for  
$i\in\mathbb{Z}^-$.
The discrete values $P_i\approx P(x_i)$ are approximate solutions.
Using the discrete values, we construct 
an approximate profile $P^\Delta(x)$ as the linear interpolation
of the discrete value on $x<0$.
%More specifically, %on the interval  $x\in[x_{i-1},i_k]$, 
%we set
%\[
%P^\Delta(x)\, \dot= \, P_{i-1}\frac{x_i - x}{x_i-x_{i-1}}    +P_i \frac{x-x_{i-1}}{x_i-x_{i-1}},
%\qquad \mbox{for}\quad x\in[x_{i-1},x_i], \quad i\le 0.
%\]
%

Given $P_i$ for $i\ge k$, we construct $P_{k-1}$ %by iteration, 
by solving a nonlinear equation
\begin{equation}\label{eq:NN}
\mathcal{G}(P_{k-1}) =0, \qquad \mbox{where}\quad 
\mathcal{G}(P_{k-1}) \;\dot=\; P_{k-1} V^\Delta (x_{k-1}; P_{k-1}) -  
P_{k} V^\Delta (x_{k}),
\end{equation}
where $V^\Delta $ is the discrete average velocity
\begin{eqnarray*}
V^\Delta (x_{k}) 
&=&  \int_{x_k}^{x_k+h} \kappa(y) v( P^\Delta(y)) w(y-x_k) \; dy ,
\\
V^\Delta (x_{k-1}; P_{k-1}) 
&=& \int_{x_{k-1}}^{x_{k-1}+h}\kappa(y)  v( P^\Delta(y;P_{k-1})) w(y-x_{k-1}) \; dy.
\end{eqnarray*}
Note that we mark the dependence on $P_{k-1}$ in the corresponding terms.
In particular, we write $P^\Delta(x;P_{k-1})$ on $x\in[x_{k-1},x_k]$ where
\[
P^\Delta(x;P_{k-1})\, \dot= \, P_{k-1}\frac{x_k - x}{x_k-x_{k-1}}    +P_k \frac{x-x_{k-1}}{x_k-x_{k-1}}.
\]

\medskip

\textbf{Step 3.} 
Assume that $P(x)$ is monotone increasing on $x\ge x_k$.
We claim that, if $\Delta x$ is sufficiently small, 
then the nonlinear equation~\eqref{eq:NN} has a unique solution
of $P_{k-1}$, satisfying $0<P_{k-1}<P_k$. 
Indeed, we have
\[\mathcal{G}(0) = - P_k V^\Delta (x_k) <0 .\]
Furthermore, if we let $P_{k-1}=P_k$, then $P^\Delta(x)$ is constant on 
the interval $x\in[x_{k-1},x_k]$, and therefore monotone increasing for $x\ge x_{k-1}$.
Then the mapping $x\mapsto v(P^\Delta(x))$ is decreasing for $x\ge x_{k-1}$. 
Since $\kappa^->\kappa^+$, then $\kappa(x)$ is also monotone decreasing.
Thus the averages satisfy 
\[V^\Delta(x_{k-1}; P_k) > V^\Delta (x_k).\]
This gives us 
\[
\mathcal{G}(P_k) =P_k \cdot \left[ V^\Delta(x_{k-1}; P_k) - V^\Delta (x_k)\right] >0 ,
\]
and we conclude that there exists a solution of $P_{k-1}$ satisfying $0<P_{k-1}<P_k$. 

In order to show that the solution is unique,  we compute the derivative
\[
%\frac{\partial G }{\partial {P_{k-1}}}
\mathcal{G}'(P_{k-1})
= V^\Delta(x_{k-1};P_{k-1}) + P_{k-1} 
\frac{\partial V^\Delta (x_{k-1};P_{k-1}) }{\partial {P_{k-1}}} %(x_{k-1};P_{k-1})
\]
where 
\begin{eqnarray*}
%\partial_{P_{k-1}} V^\Delta
\frac{\partial V^\Delta (x_{k-1};P_{k-1})}{\partial {P_{k-1}}}
&=& 
\kappa^- \int_{x_{k-1}}^{x_k} %\partial_{P_{k-1}}
\frac{\partial v}{\partial P_{k-1}} 
\left(P_{k-1}\frac{x_k - y}{\Delta x}    +P_k \frac{y-x_{k-1}}{\Delta x} \right) \cdot
w(y-x_{k-1})\; dy\\
&=& \kappa^- \int_{x_{k-1}}^{x_k}  v'\left(P^\Delta(y)\right) \frac{x_k - y}{\Delta x}  \cdot
w(y-x_{k-1})\; dy
~=~
-\mathcal{O}(1) \Delta x.
\end{eqnarray*}
Therefore, for $\Delta x$ sufficiently small we have
\[
 \mathcal{G}'(P_{k-1})  = V(x_{k-1}) - P_{k-1}  \mathcal{O}(1) \Delta x >0.
\]
Thus the mapping $P_{k-1}\mapsto \mathcal{G}$ is monotone and the root is unique,
proving the claim.

\medskip

\textbf{Step 4.}  
Set $P_0=\cW(0)$.  By the above construction, 
we generate $P_i$ for $i=-1,-2, \cdots$, with $0< P_i < P_{i+1} < P_0$, 
and thus a positive monotone approximate solution $P^\Delta(x)$ on $x\le 0$.
%Denote the discrete flux 
%\[
%F^\Delta(x_i) \;\dot=\; P^\Delta(x_i) V^\Delta (x_i), \qquad i=0,-1,-2,\cdots.
%\]
%By the initial data on $x\ge 0$, we have $F^\Delta(x_0) = \bar f$.  
By the construction, we have 
\[
%F^\Delta(x_i) 
P^\Delta(x_i) V^\Delta (x_i) = \bar f, \qquad \forall i \le 0.
\]
Taking the limit $\dx\to 0$, the sequence of approximate monotone 
solutions $P^\Delta(x)$
converges to a limit function $P(x)$, which satisfies 
$ P(x) V(x) = \bar f$ for every $x\le 0$. 
Thus, $P(x)$ is a solution to~\eqref{C3}, and 
$\lim_{x\to-\infty} P(x) = \rho^-$. 

Finally, for each ``initial data'' (i.e., some horizontal shift of $\cW(x)$) there exists a unit 
profile $P(x)$, therefore we obtain infinitely many profiles. 
\end{proof}

Sample profiles are plotted in Figure~\ref{fig:C1}. We observe that they do not cross each 
other. The ordering property in Corollary~\ref{cor2} holds for these profile, 
with a very similar proof. 

\medskip

\textbf{Stability.} 
In the case when $\kappa(x)\equiv \bar\kappa$ is a constant function, 
the existence and well-posedness of solution for (M2) is established in~\cite{Friedrich2018},
through convergence of approximate solutions generated by a Godunov-type scheme. 
Unfortunately, when $\kappa(x)$ is discontinuous, no existence result
is yet available.
Assuming that the solutions exist, we show that they converge to the 
traveling wave profile as $t$ grows, under mild assumptions on the initial data.

\begin{definition}\label{def:M2}
We say that $\rho(t,x)$ is a solution of~\eqref{eq:CL2} if $\rho(t,x)$ is continuous
and satisfies the equation~\eqref{eq:CL2}
for all $x<0$ and $x>0$, and the connecting condition
\begin{equation}\label{CCC2}
\rho_x(0+,t)-\rho_x(0-,t) =  \frac{(\kappa^+-\kappa^-) \rho(0,t)  v(\rho(0,t)) w(0) }{\kappa^+ \int_0^h v(\rho(y,t)) w(y) dy}
\qquad \forall t>0. 
\end{equation}
%for every $t>0$. 
\end{definition}

With a very similar argument as those for Case A1, 
one can prove that the profiles $P(x)$ are time asymptotical 
limit for the Cauchy problem of~\eqref{eq:CL2}.  We omit the details of the proof, 
and state  the following stability Theorem.

\begin{theorem}\label{th:C3n}
Let $\rho(x,t)$ be the solution to the Cauchy problem for~\eqref{eq:CL2} 
with initial data $\rho(x,0)$, with $\kappa^- >\kappa^+$. 
Assume that the initial data satisfies the connecting condition~\eqref{CCC2}
and 
$P^\flat (x) \le \rho(x,0) \le P^\sharp(x)$,
for some profiles $P^\flat$ and $P^\sharp$. 
Then, the solution $\rho(x,t)$ converges to a profile $P(x)$ as $t\to\infty$.
\end{theorem}

We perform a numerical simulation with Riemann initial data, and plot the result
in Figure~\ref{fig:LFC1}. We observe that the solution $\rho(x,t)$ quickly
approaches a profile, confirming Theorem~\ref{th:C3n}.
We remark that the initial Riemann data actually does not satisfy the assumptions
in Theorem~\ref{th:C3n}, and yet we still observe stability.
This indicates that the basin of attraction is probably larger than what is stated in 
Theorem~\ref{th:C3n}.

\begin{figure}[htbp]
\begin{center}
\includegraphics[width=7.5cm,height=4.5cm,clip,trim=3mm 1mm 11mm 6mm]{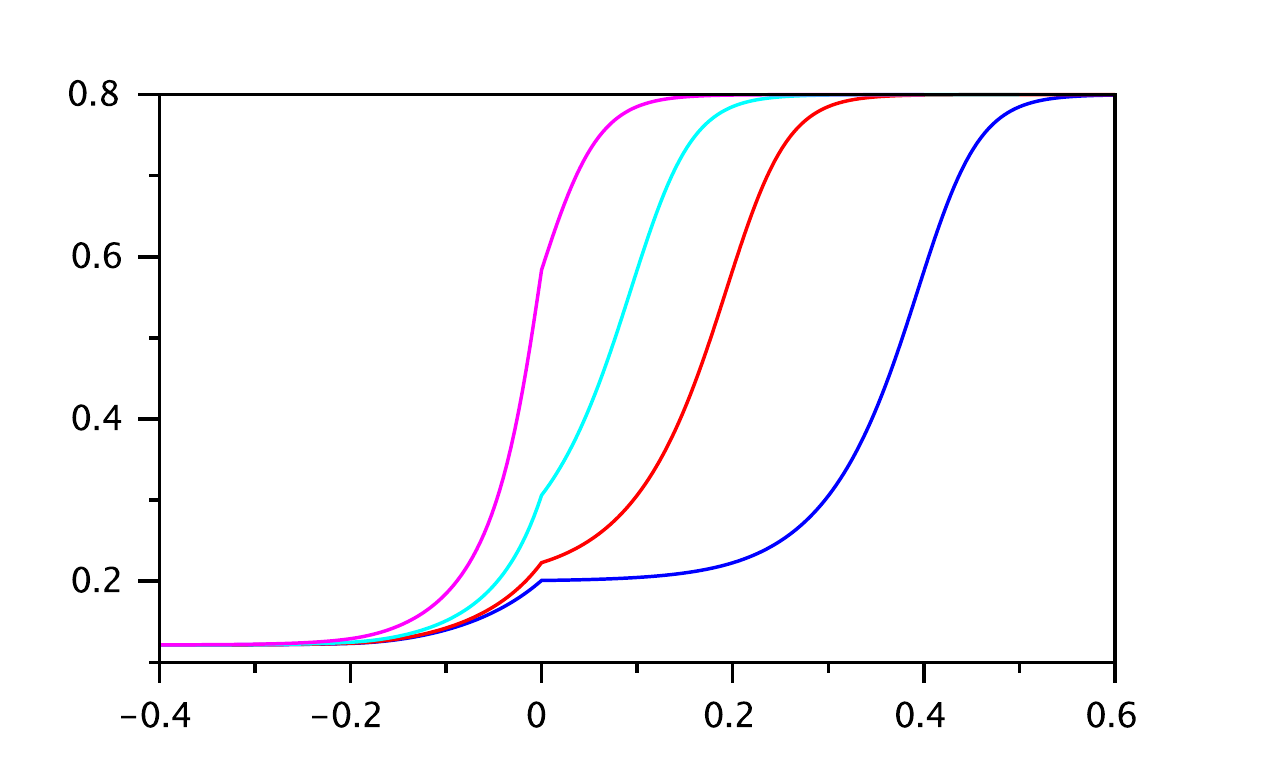}
\caption{Sample traveling wave for Case C1.}
\label{fig:C1}

\bigskip

\includegraphics[width=13cm,clip,trim=0mm 1mm 5mm 1mm]{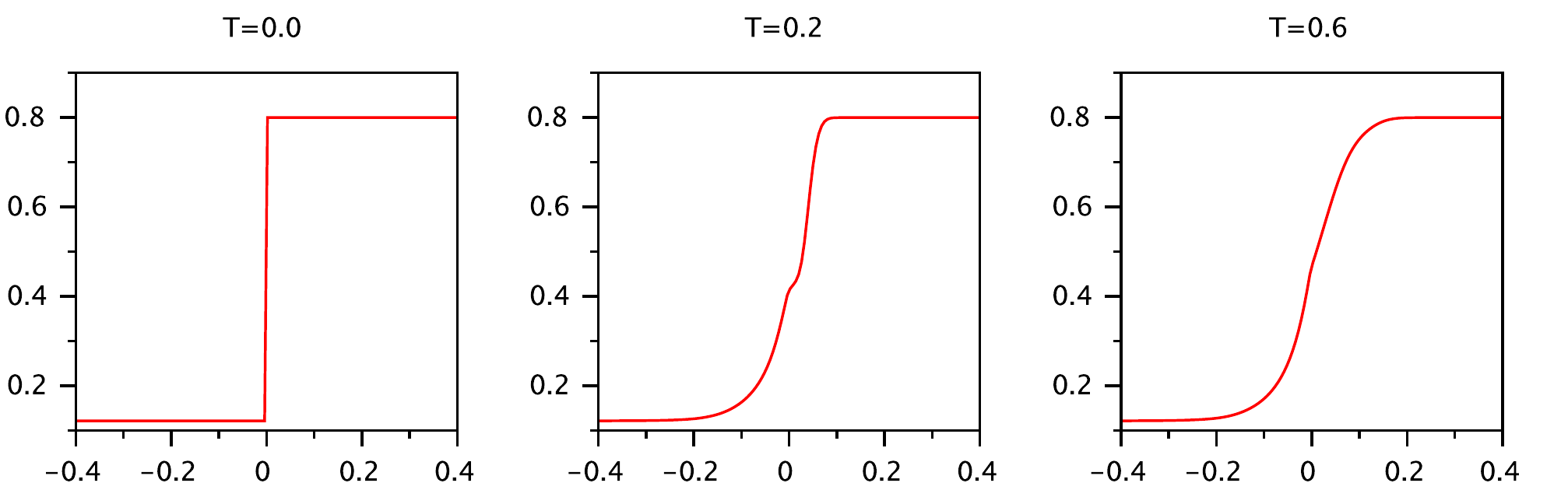}
\caption{Numerical simulation for the PDE model with Riemann initial data
for Case C1.}
\label{fig:LFC1}
\end{center}
\end{figure}

%%%%%%%%%%%%
\subsubsection{Case C2, Case C3, and Case C4}

For Case C2 we have $\rho^- <\rho^+ \le \hat\rho$.
Since $\rho^+$ is an unstable asymptote for $x\to\infty$, 
we must have the constant solution $P(x)\equiv \rho^+$ for $x\ge 0$.
A unique profile $P(x)$ can be obtained by solving this initial value problem 
backward in $x$, for $x<0$. 
This unique profile  is monotone increasing, see a sample graph in Figure~\ref{fig:C2}. 
Finally, the profile is not a local attractor for solutions of~\eqref{eq:CL2}, since 
$\rho^+$ is an unstable asymptote. 
This is further confirmed by the numerical simulation in Figure~\ref{fig:LFC2},
where we observe oscillation entering the region $x>0$ and persisting in time.

For Case C3 with $\hat\rho \le \rho^+ < \rho^-$ and 
Case C4 with $\rho^- > \hat\rho >\rho^+$,  there are no profiles, 
similar to the results for Case A3 and Case A4.  
Numerical simulations with Riemann initial data are given in Figure~\ref{fig:LFC3} 
for Case C3, and in Figure~\ref{fig:LFC4} for Case C4. 
In both cases, we observe that oscillations form around $x=0$ and enter the region
$x>0$ and/or $x<0$, and the oscillations persist as $t$ grows.

\begin{figure}[htbp]
\begin{center}
\includegraphics[width=7.5cm,height=4cm,clip,trim=3mm 1mm 11mm 6mm]{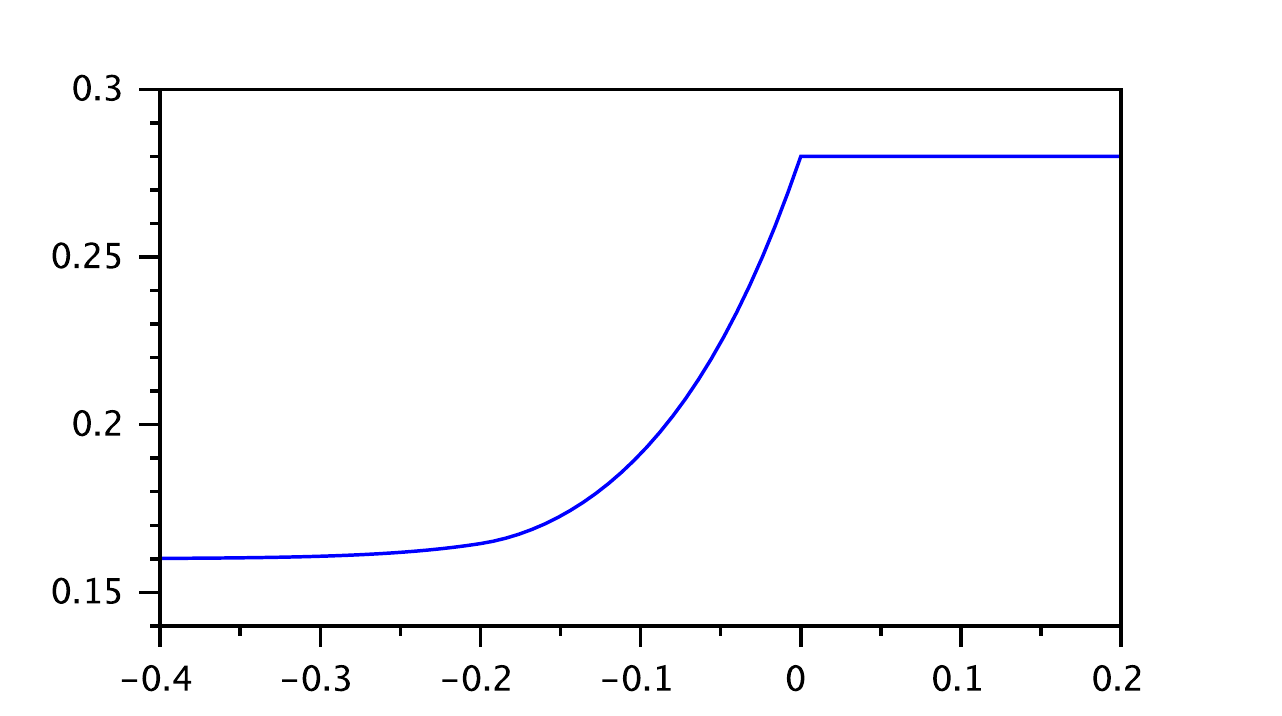}
\caption{Sample traveling wave for Case C2.}
\label{fig:C2}

\bigskip

\includegraphics[width=13cm,clip,trim=0mm 1mm 5mm 1mm]{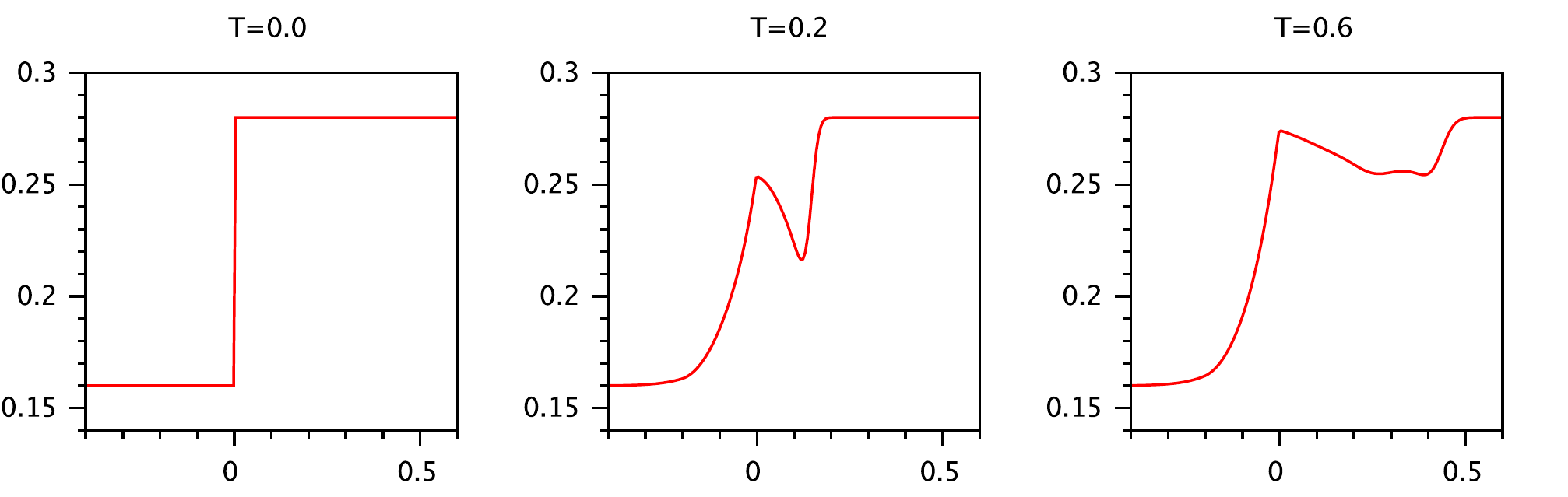}
\caption{Numerical simulation for the PDE model with Riemann initial data for Case C2.}
\label{fig:LFC2}

\includegraphics[width=13cm,clip,trim=0mm 1mm 5mm 1mm]{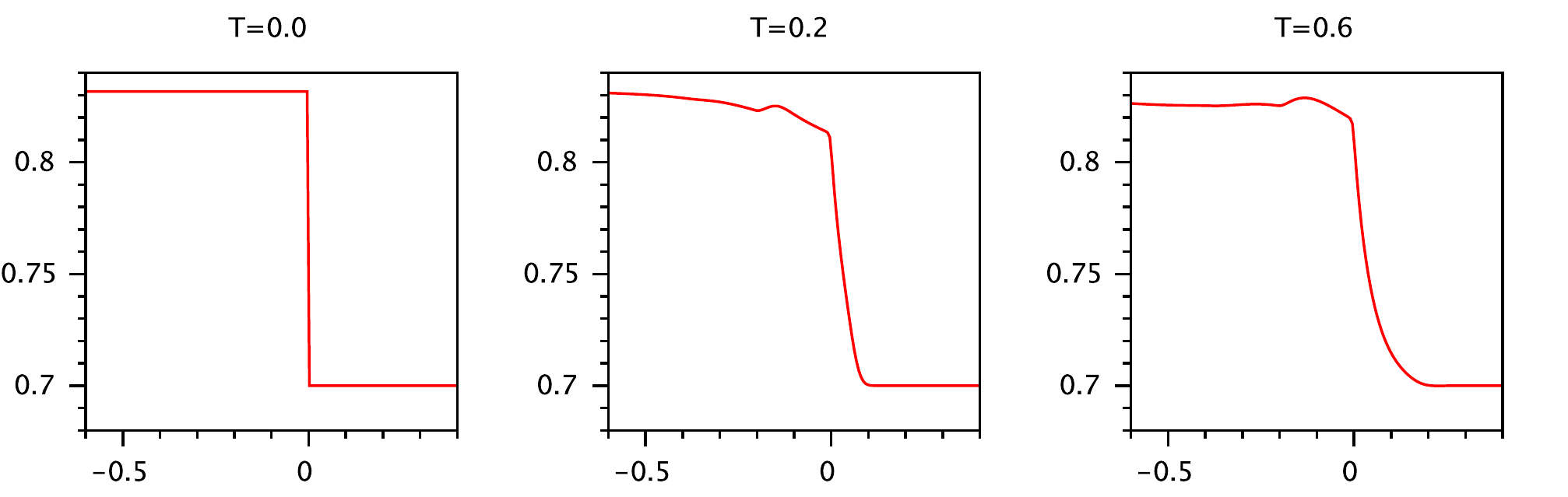}
\caption{Numerical simulation for the PDE model with Riemann initial data for Case C3.}
\label{fig:LFC3}

\bigskip

\includegraphics[width=13cm,clip,trim=0mm 1mm 5mm 1mm]{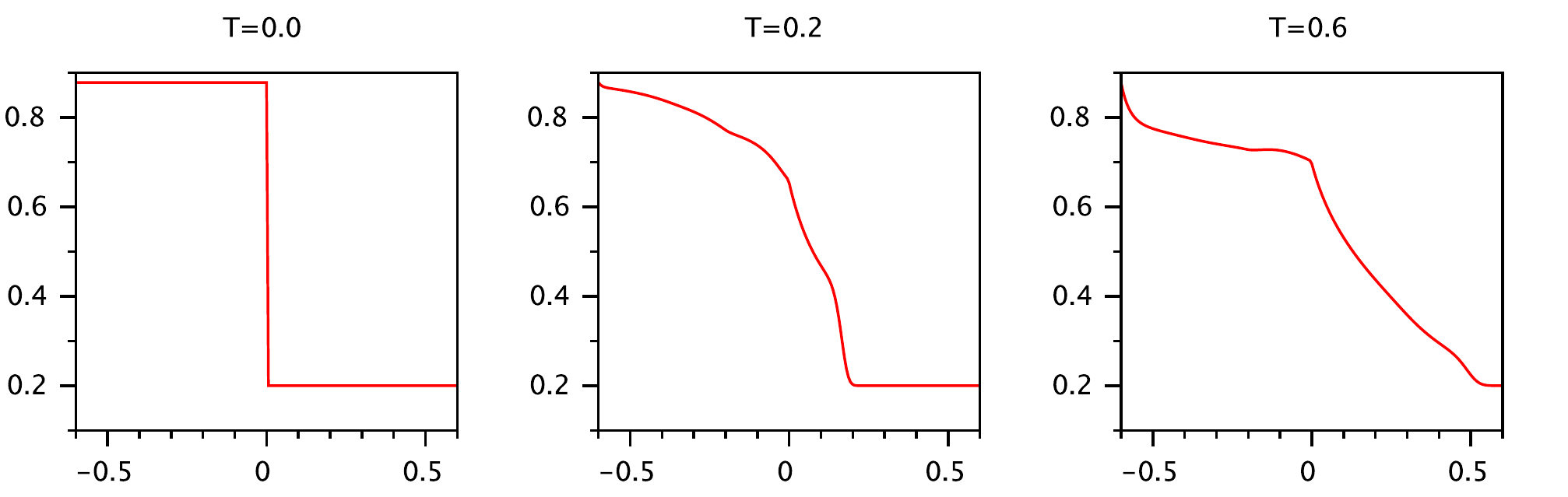}
\caption{Numerical simulation for the PDE model with Riemann initial data for Case C4.}
\label{fig:LFC4}
\end{center}
\end{figure}

%%%%%%%%%%%%%%%%%%%%%%%
\subsection{Case D:  $\kappa^-<\kappa^+$}
This is the counter part for Case B, and similarly we have 4 sub-cases
which we discuss in detail in the following sub sections.

%%%%%%%%%%%%
\subsubsection{Case D1:  $\rho^- < \hat\rho < \rho^+$}

%This is the counter part for Case B1.  
%As usual, we neglect the trivial case $\rho^-=0,\rho^+=1$.
%We have a similar Theorem.

\begin{theorem}\label{th:D1}
Let $\kappa^- < \kappa^+$, and let $\rho^-,\rho^+$ satisfy
\[ 0<\rho^- < \hat\rho < \rho^+<1, \qquad f^-(\rho^-)=f^+(\rho^+) . \]
Then there exist infinitely many solutions for~\eqref{C3} which satisfy the 
``boundary conditions''
\[ \lim_{x\to-\infty} P(x) = \rho^-,\qquad \lim_{x\to+\infty} P(x) =\rho^+.\]
%All profiles are non-intersecting. 
%Furthermore, they are monotone increasing on $x>0$. 
\end{theorem}

\begin{proof}
The proof follows a similar line of argument as those in the proof of Theorem~\ref{tm:B1}. 

\textbf{Step 1.} On $x\ge 0$, the profile can 
be either the constant function $P(x)\equiv \rho^+$ or some horizontal shift of $\cW(x)$. 
Unfortunately, the constant function $P(x)\equiv \rho^+$ on $x \ge 0$ 
does not work.  Indeed, by the connecting condition~\eqref{connectC} 
we have 
\[
P'(0-) = P'(0+) - (\kappa^+-\kappa^-) \rho^+ w(0) <0.
\]
As $x<0$ get smaller, $P(x)$ grows monotonically, and reaches the value $1$ at some finite
$\hat x$, and solution can not continue for $x<\hat x$. 

Thus, $P(x)$ is some horizontal shift of $\cW$. 
At $x=0$, $P(0)$ takes value between $\rho_1$ and $\rho^+$, where 
$\rho_1 < \hat\rho<\rho^+$ and $f^+(\rho_1) =f^+(\rho^+) \;\dot=\; \bar f$.

\medskip 

\textbf{Step 2.}  Consider the initial value problem with $P(x)\equiv  \rho_1$ on $x\ge 0$
where $\rho_1$ is defined in Step 1.
By a similar argument as in Step 4 of the proof for Theorem~\ref{tm:B1}, 
the solution is monotone decreasing and approaches $\rho^-$ as $x\to-\infty$.
We call this profile $P^\flat(x)$.
By the ordering property, all profiles $P(x)$ lie above $P^\flat(x)$, 
therefore $P^\flat(x)$ serves as a lower envelope.

\medskip 

\textbf{Step 3.} Let $P(x)$ be a profile such that $P(x)$ equals some horizontal shift 
of $\cW(x)$ on $x\ge 0$. 
As $P(0)$ varies from $\rho_1$ to $\rho^+$, the profile might lose monotonicity
and becomes oscillatory. By a similar argument as in Step 5 of the proof for 
Theorem~\ref{tm:B1}, we have, for $x<-h$,
\[
V'(x) = -\kappa^- \int_0^h [v(P(x+s)) - v(P(x))] w'(s) \; ds.
\]
Thus, if $V'(x')=0$ for some $x'$, then on $x\in[x',x'+h]$ we have either 
$P(x)\equiv P(x')$, or $P(x)$ oscillates around $P(x')$. 
Following the same argument as in the rest of Step 5 of the proof for 
Theorem~\ref{tm:B1}, we conclude 
that there exists an upper envelope $P^\sharp(x)$ for all the profiles
that approaches $\rho^-$ as $x\to-\infty$.
By continuity and the ordering of the profiles, there exist infinitely many
profiles between $P^\flat(x)$ and $P^\sharp(x)$. 
\end{proof}

Sample profiles are given in Figure~\ref{fig:D1}. 
Using a  similar argument as for Theorem~\ref{th3} for Case A1, 
these profiles are time asymptotic limits for the solutions of the the Cauchy problem
of~\eqref{eq:CL2}. We omit the details of the proof. 
A numerical simulation is presented in Figure~\ref{fig:LFD1}, with Riemann data,
where we observe this asymptotic behavior as $t$ grows.

\begin{figure}[htbp]
\begin{center}
\includegraphics[width=7.5cm,clip,trim=3mm 1mm 11mm 6mm]{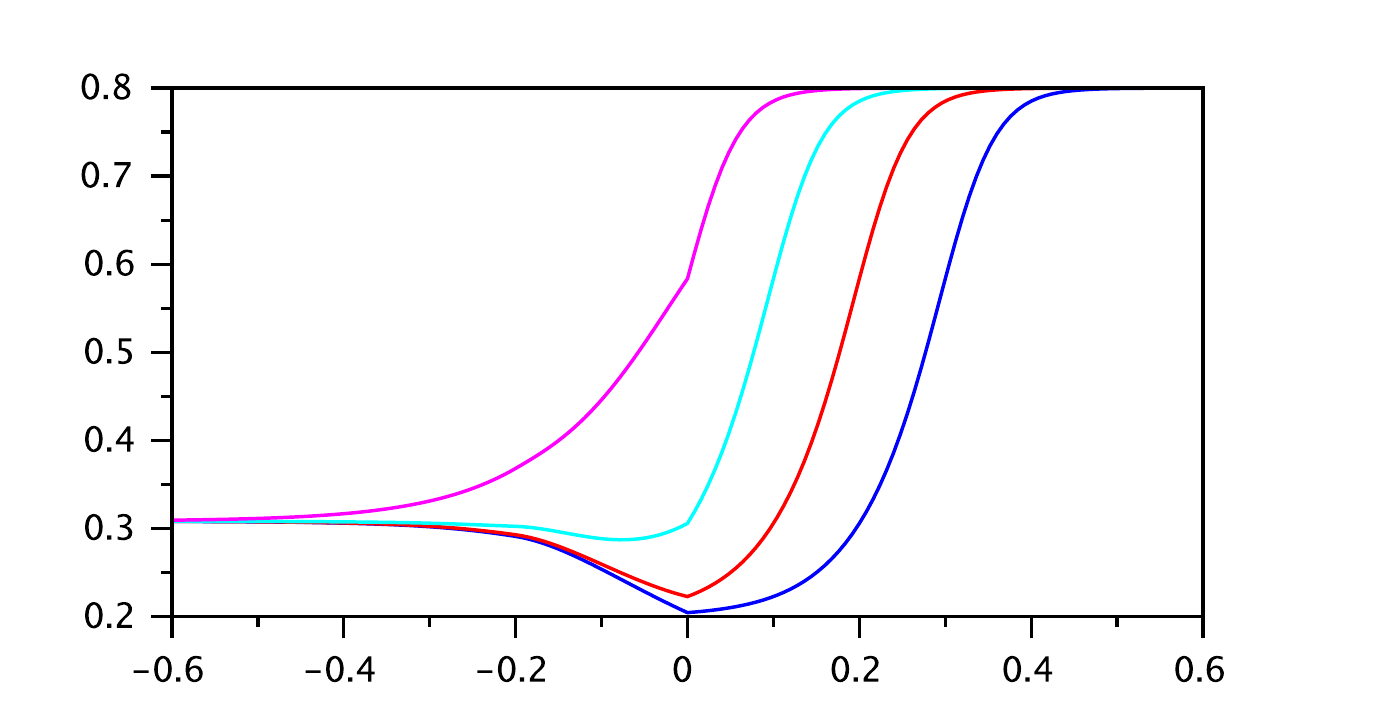}
\caption{Sample traveling wave for Case D1.}
\label{fig:D1}

\bigskip

\includegraphics[width=13cm,clip,trim=0mm 1mm 5mm 1mm]{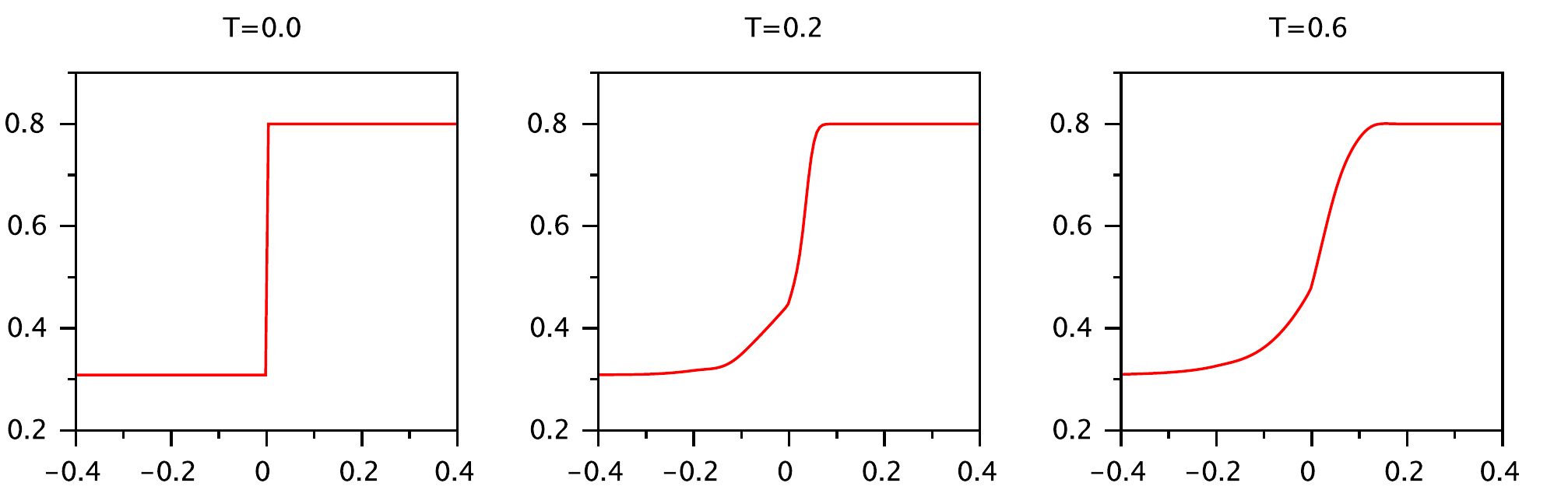}
\caption{Numerical simulation for the PDE model with Riemann initial data
for Case D1.}
\label{fig:LFD1}
\end{center}
\end{figure}

%%%%%%%%%%%%
\subsubsection{Case D2, Case D3, and Case D4}

For Case D2 we have  $\rho^+ < \rho^- \le \hat\rho $. 
Since $\rho^+$ is an unstable asymptote for $x\to\infty$, 
we must have  with $P(x)\equiv \rho^+$  on $x\ge 0$. 
The unique profile $P(x)$ is constructed by solving this initial value problem 
backward in $x$ on $x<0$. 
The profile is monotone decreasing, see a sample plot in Figure~\ref{fig:D2}.
However, the profile is not an attractor for solutions 
of the conservation law~\eqref{eq:CL2},
as also indicated by the numerical simulation in Figure~\ref{fig:LFD2}.
We observe that oscillation forms around $x=0$, then enters the region $x>0$, and it persists
in time. 

For  Case D3 we have  $ \hat\rho \le \rho^- < \rho^+$,
while for  Case D4 we have  $\rho^+ < \hat\rho < \rho^-$.
For both cases, there are no profiles, similar to the results in Case B3 and Case B4.
Numerical simulation results are presented in Figure~\ref{fig:LFD3}
for Case D3 and in Figure~\ref{fig:LFD4}
for Case D4.  In both cases, we observe oscillations that enter the region $x<0$ 
and/or $x>0$, and they persist in time. 

\begin{figure}[htbp]
\begin{center}
\includegraphics[width=7.5cm,clip,trim=4mm 1mm 12mm 7mm]{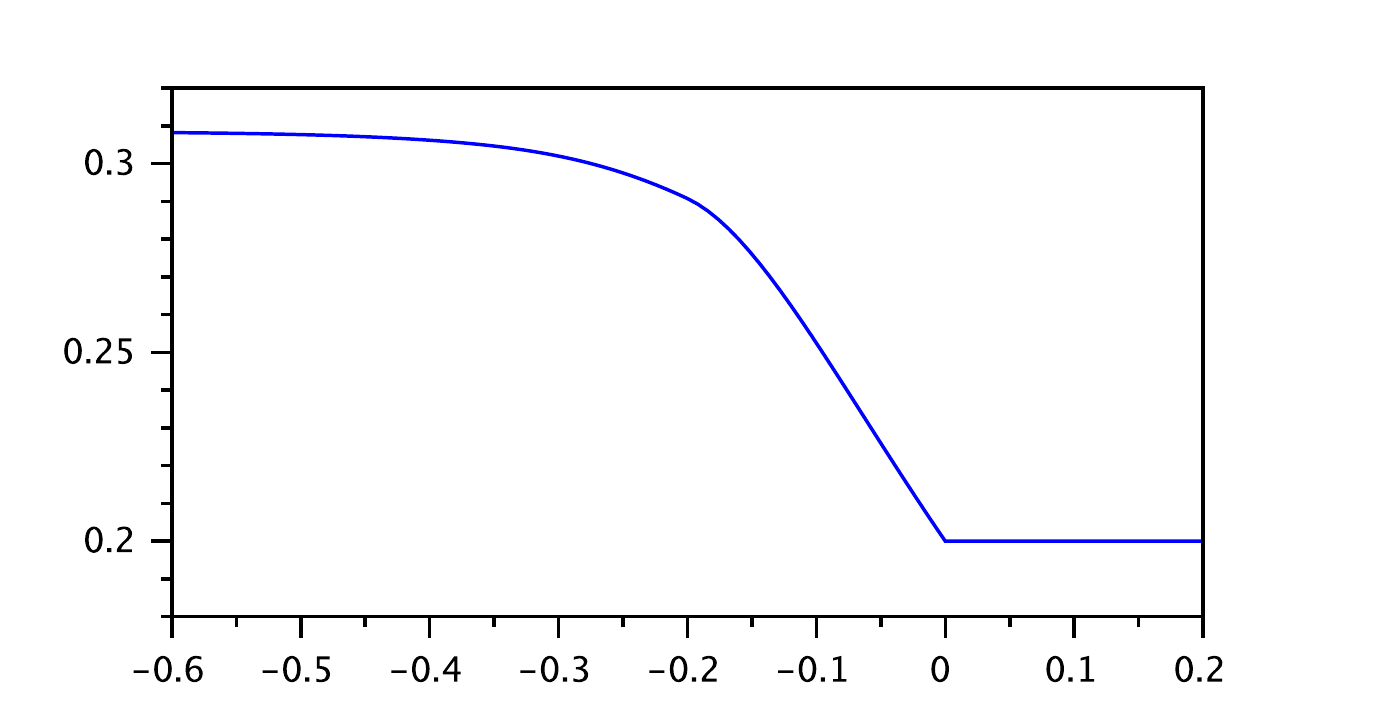}
\caption{Sample traveling wave for Case D2.}
\label{fig:D2}

\bigskip

\includegraphics[width=12.5cm,clip,trim=0mm 1mm 6mm 1mm]{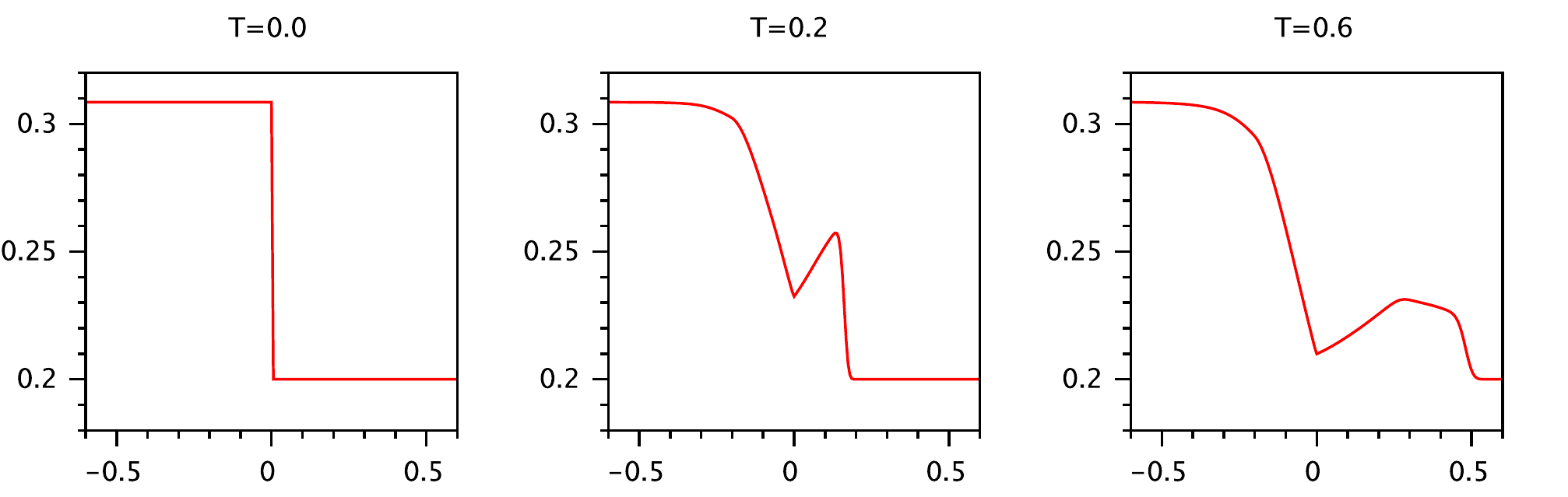}
\caption{Solution of Riemann problem for Case D2.}
\label{fig:LFD2}

\bigskip

\includegraphics[width=12.5cm,clip,trim=0mm 1mm 6mm 1mm]{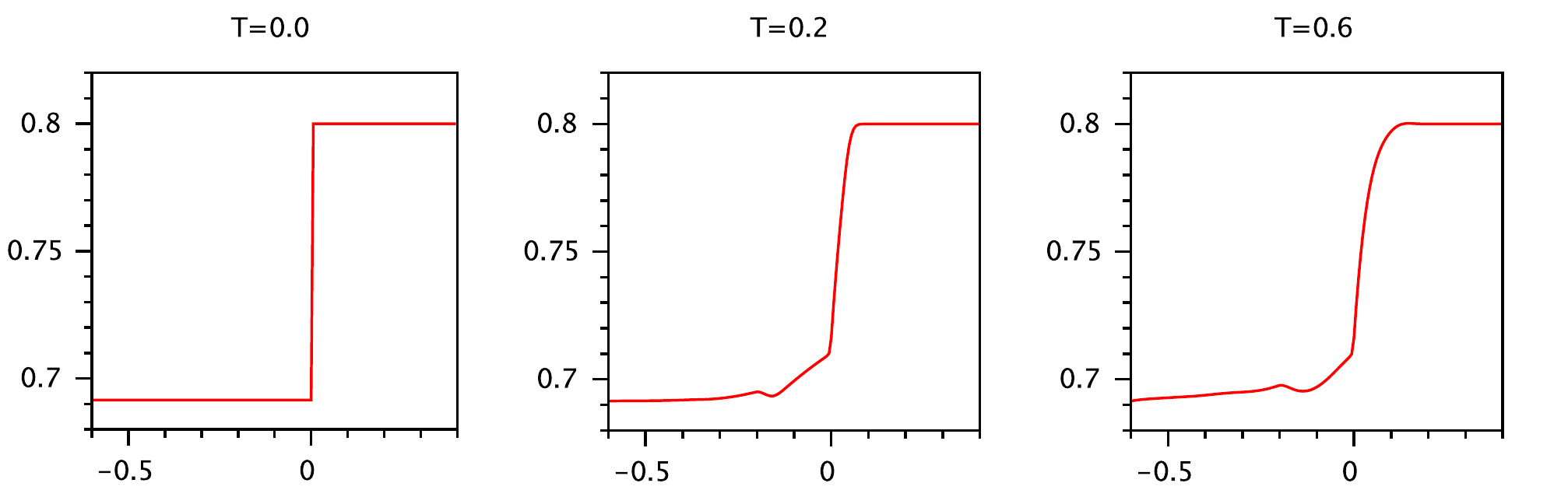}
\caption{Solution of Riemann problem for Case D3.}
\label{fig:LFD3}

\bigskip

\includegraphics[width=12.5cm,clip,trim=0mm 1mm 6mm 1mm]{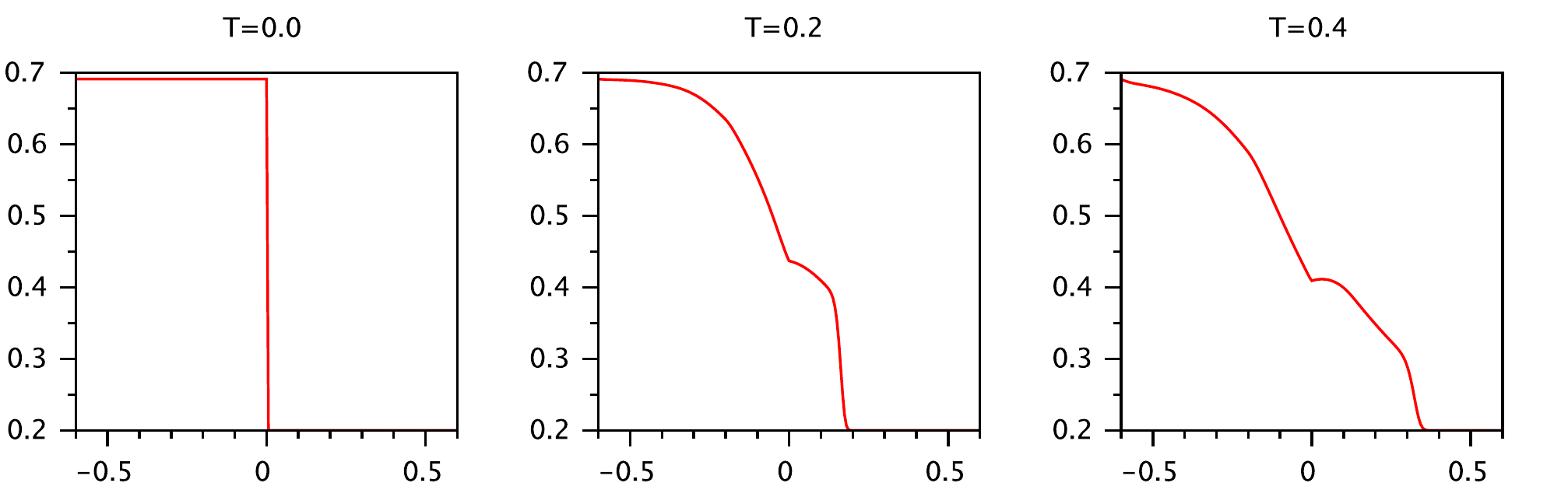}
\caption{Solution of Riemann problem for Case D4.}
\label{fig:LFD4}
\end{center}
\end{figure}

%%%%%%%%%%%%%%%%%%
\section{Concluding remarks}\label{sec:fr}

We study traveling wave profiles for two nonlocal PDE models for traffic flow
with rough road condition. 
For all possible cases, we show that there can exist 
infinitely many traveling wave profiles, a unique profile, or no profiles at all,
depending on the jump in speed limit and the limits $\rho^-,\rho^+$. 
The stability of these profiles also vary from case to case.

Formally, the non-local PDE models are the macroscopic limits of the corresponding 
non-local particle models, referred to as the follow-the-leaders (FtLs) model. 
In the case where $\kappa(x)\equiv 1$, 
the existence of traveling waves and their convergence to the traveling waves
for the PDE models were provided in~\cite{RidderShen2018}.
It would be interesting to establish a similar result for the FtLs model
where $\kappa(x)$ has a jump. 
Details will come in  a forthcoming work. 

Codes for all the numerical simulations in this paper can be found at: 

\centerline{ \texttt{http://www.personal.psu.edu/wxs27/TrafficNLRough/}}

\end{document}